\documentclass[11pt]{amsart}

\usepackage{graphicx,epsfig}
\usepackage[dvipsnames]{xcolor}
\usepackage{amsfonts,amsmath,amssymb,amsthm}
\usepackage[normalem]{ulem}
\usepackage{comment,constants}
\usepackage{hyperref}
\usepackage{mathtools}
\usepackage{multirow}

\newtheorem{thm}{Theorem}[section]
\newtheorem{prop}[thm]{Proposition}
\newtheorem{lem}[thm]{Lemma}
\newtheorem*{asm*}{Assumptions}
\newtheorem{asm}{Assumption}

\theoremstyle{remark}
\newtheorem{rem}[thm]{Remark}
\newtheorem*{rem*}{Remark}

\theoremstyle{definition}

\newcommand{\ra}{\rightarrow}

\newcommand{\R}{\mathbb R}     % For Real numbers
\newcommand{\Z}{\mathbb Z}     % For Integers
\renewcommand{\a}{\alpha}

\renewcommand{\d}{\delta}

\newcommand{\e}{\varepsilon}

\newcommand{\s}{\sigma}
\renewcommand{\th}{\theta}
\renewcommand{\epsilon}{\varepsilon}
\renewcommand{\k}{\kappa}

\newcommand{\bigo}{\mathcal{O}}

\newcommand{\fl}[1]{\lfloor #1 \rfloor}  % Floor function
\DeclarePairedDelimiter{\ceil}{\lceil}{\rceil}

\newconstantfamily{c}{symbol=c}
\newconstantfamily{CU}{symbol=C_U}

\newcommand{\w}{\omega}              % Shortcut for \omega
\renewcommand{\P}{\mathbb{P}}        % Annealed Probability
\newcommand{\E}{\mathbb{E}}          % Annealed Expectation
\newcommand{\vp}{\mathrm{v}_0}       % Limiting velocity
\DeclareMathOperator{\Var}{Var}
\DeclareMathOperator{\Cov}{Cov}

\title[Quenched CLT rates of convergence]{Quenched central limit theorem rates of convergence for one-dimensional random walks in random environments}
\author{Sung Won Ahn}
\address{Sung Won Ahn\\Roosevelt University\\Department of Mathematics and Actuarial Science\\430 S. Michigan Ave.\\Chicago, IL 60605\\USA} % add more detailed address
\email{sahn02@roosevelt.edu}
\author{Jonathon Peterson}
\address{Jonathon Peterson\\Purdue University\\Department of Mathematics\\150 N University Street\\West Lafayette, IN  47907\\USA}
\email{peterson@purdue.edu}
\urladdr{http://www.math.purdue.edu/~peterson}
\thanks{J. Peterson was partially supported by NSA grant H98230-15-1-0049.}


\begin{document}

\begin{abstract}
Unlike classical simple random walks, one-dimensional random walks in random environments (RWRE) are known to have a wide array of potential limiting distributions. Under certain assumptions, however, it is known that CLT-like limiting distributions hold for the walk under both the quenched and averaged measures. 
We give upper bounds on the rates of convergence for the quenched central limit theorems for both the hitting time and position of the RWRE with polynomial rates of convergence that depend on the distribution on environments. 
\end{abstract}

\maketitle

\section{Introduction}

If $\{\xi_k\}_{k\geq 1}$ is an i.i.d.\ sequence of zero mean random variables with finite variance $\s^2 = E[\xi_1^2]$, then the central limit theorem implies that the rescaled sum $Z_n = \frac{1}{\s\sqrt{n}} \sum_{k=1}^n \xi_k$ converges in distribution to a standard Gaussian random variable. That is, $F_n(x) = P(Z_n \leq x) \to \Phi(x)$ where $\Phi$ is the c.d.f.\ of the standard normal distribution. The central limit theorem, however, offers no quantitative bounds on the rate of convergence of $F_n$ to $\Phi$ and in fact additional moment assumptions are needed to obtain such rates of convergence. 
The classical Berry-Esseen Theorem \cite{bBET,eBET} states that there is a universal constant $A_1<\infty$ such that if $\xi_1$ has finite third moment then 
\[
 \| F_n - \Phi\|_\infty = \sup_{x\in\R} \| F_n(x) - \Phi(x) \| \leq \frac{A_1 E[|\xi_1|^3]}{\s^3\sqrt{n}}, \quad \forall n\geq 1. 
\]
More generally, one can obtain slower rates of convergence under weaker moment assumptions. In particular, it follows from \cite{kBET} that for any $\d \in (0,1]$ there exists a universal constant $A_\d <\infty$ such that if $\xi_k$ has finite $(2+\d)$-th moment then
\[
 \| F_n - \Phi\|_\infty \leq \frac{A_\d E[|\xi_1|^{2+\d}]}{\s^{2+\d} n^{\d/2}}, \quad \forall n\geq 1. 
\]

In this paper we will be concerned with obtaining Berry-Esseen like rates of convergence for central limit theorems arising in one-dimensional random walks in random environments. A random walk in a random environment (RWRE) is a simple model for random motion in a non-homogeneous environment. The class of models that may be considered RWRE is quite large, but we will be concerned here with the case of (nearest-neighbor) one-dimensional RWRE. In this model, a \emph{random environment} is a random sequence $\w = \{\w_x \}_{x\in \Z} \in [0,1]^\Z$ which can be used to determine the transition probabilities for a Markov chain on $\Z$ with steps of size $\pm 1$. In particular, given an environment $\w$ and a starting point $x \in \Z$ we will denote by $P_\w^x$ the law of a Markov chain $\{X_n\}_{n\geq 0}$ defined by $P_\w^x(X_0 = x) = 1$ and 
\[
 P_\w^x\left(X_{n+1} = y+1 \, | \, X_n=y \right) 
=  1 - P_\w^x\left(X_{n+1} = y-1 \, | \, X_n=y \right) = \w_y. 
\]
%Usually we will start the walk at $x=0$ and so we will use $P_\w$ instead of $P_\w^0$ in this case. 
The distribution $P_\w^x$ of  the walk in a fixed environment is called the \emph{quenched} law of the RWRE. 
%In general, it is assumed that the environment $\w = \{\w_x\}_{x\in \Z}$ is ergodic, and if $P$ denotes the distribution of the environment $\w$ then 
If $P$ denotes the probability distribution of the environment $\w$, then by averaging the quenched $P_\w^x$ law with respect to $P$ we obtain the \emph{averaged} (or \emph{annealed}) law of the RWRE:
\[
 \P^x(\cdot) = E\left[ P_\w^x(\cdot) \right]. 
\]
Expectations with respect to the quenched and averaged laws of the walk are denoted by $E_\w^x$ and $\E^x$, respectively. 
Usually the walk will be started at $X_0=0$ and we will use $P_\w$ and $\P$ to denote the quenched and averaged laws in this case and corresponding expectations by $E_\w$ and $\E$, respectively.
Finally, variances under the quenched measure $P_\w$ will be denoted by $\Var_\w$; that is $\Var_\w(Z) = E_\w[Z^2] - E_\w[Z]^2$.  

While RWREs are a rather simple generalization of classical simple random walks, the behaviors of RWREs can be quite different than what is known for simple random walks. 
%Indeed, RWREs can be directionally transient with sublinear asymptotic speed \cite{sRWRE} and can have limiting distributions which are non-Gaussian under non-diffusive scaling \cite{kksStable,sRecurrent}. 
For instance, if the distribution on environments is such that the walk is recurrent then (under rather tame additional assumptions)
% $X_n/(\log n)^2$ converges in distribution to a non-Gaussian distribution \cite{kksStable}. 
the position of the walk converges in distribution to a non-Gaussian distribution when scaled by $(\log n)^2$ rather than the diffusive $\sqrt{n}$ scaling in classical simple random walks \cite{sRecurrent}.
Transient RWREs can also exhibit a variety of non-Gaussian limiting distributions under non-diffusive scalings \cite{kksStable,mrzStable}, but in this paper we will be assuming conditions under which it is known that CLT-like limiting distributions hold. 

The first assumption that we will be making in this paper is that the environments are i.i.d.
\begin{asm}\label{asm:iid}
 The distribution $P$ on environments is such that $\w = \{\w_x\}_{x\in \Z}$ is i.i.d. 
\end{asm}
\noindent
For our second main assumption we will need to introduce some additional notation. First, let 
\[
 \rho_x = \frac{1-\w_x}{\w_x}, \quad x \in \Z. 
\]
Many of the known results for RWREs can be stated in terms of the distribution of this ratio of transition probabilities. For instance, under Assumption \ref{asm:iid} the RWRE is transient to the right if $E[\log \rho_0] < 0$ and the limiting speed $\vp = \lim_{n\to\infty} X_n/n$ is positive if and only if $E[\rho_0] < 1$ \cite{sRWRE}.
In this paper we will be making the following assumption regarding the moments of the random variable $\rho_0$. 
%In this paper, we will be assuming the following assumption which is equivalent to the assumption that $E[\rho_0^{2+\d}] < 1 $ for some $\d>0$. 
%we will be interested in the following parameter $\k$ 
%\begin{align*}
% \k = \sup\{ p > 0: E[\rho_0^p] < 1 \}. 
%\end{align*}
\begin{asm}\label{asm:k2}
 %If $\rho_x = \frac{1-\w_x}{\w_x}$ for $x\in \Z$, then $E[\rho_0^{2+\d}] < 1$ for some $\d>0$. 
%$E[\rho_0^{2+\d}] < 1$ for some $\d>0$, where $\rho_x = \frac{1-\w_x}{\w_x}$. 
 $\k := \sup\{ p > 0: E[\rho_0^p] < 1 \} > 2$ (or equivalently $E[\rho_0^{2+\d}] < 1$ for some $\d>0$). 
\end{asm}
%Note that Assumption \ref{asm:k2} is equivalent to the statement that $E[\rho_0^{2+\delta}] < 1$ for some $\d>0$. 
Since $t \mapsto E[\rho_0^t] = E[e^{t \log \rho_0}]$ is the moment generating function of $\log \rho_0$ and is therefore a convex function in $t$, it follows from Assumption \ref{asm:k2} that $E[\log \rho_0] < 0$ (that is the walk is transient to the right) and that 
\begin{equation}\label{rpdef}
 r_p := E[\rho_0^p] < 1 \quad \text{for all } p \in (0,\k). 
\end{equation}
In particular, this implies that $r_1 = E[\rho_0] < 1$ so that the speed $\vp$ of the walk is positive. 

It should be noted that under rather mild additional assumptions it holds that
% $r_\kappa = E[\rho_0^\kappa] = 1$. 
\begin{equation}\label{oldkdef}
% r_\kappa = 
E[\rho_0^\kappa] = 1.
\end{equation}
In fact, in a number of previous results in RWRE the parameter $\k$ is defined as the unique positive solution to equation \eqref{oldkdef}. 
%is used to define the parameter $\k>0$. 
For instance, the parameter $\k$ defined this way is used in studying limiting distributions of transient RWRE \cite{kksStable,pzSL1,p1lsl2,psWQLTn,psWQLXn,estzWQL,dgWQL}, identifying the subexponential rate of decay of certain large deviation probabilities \cite{dpzTE1D,gzQSETE,apOQSA}, and identifying the maximal displacement of large ``bridges'' of RWRE \cite{gpRWREBridges}. A number of these results assume additional technical conditions (e.g., $E[\rho_0^\kappa \log \rho_0] < \infty$ and the distribution of $\log \rho_0$ is non-lattice) to obtain certain precise tail asymptotics, but we will not need these conditions nor the slightly more restrictive definition of $\k$ in \eqref{oldkdef}.

The relevance of the parameter $\k$ to the limiting distributions of transient RWRE comes from the fact that $\k$ determines what moments of the hitting times of the RWRE are finite (c.f. Lemma \ref{lem:taumtail} below); in particular, hitting times have finite second moment if $\k>2$. 
The limiting distributions under the averaged measure $\P$ for transient RWRE in \cite{kksStable} show that CLT-like limiting distributions hold only when $\k > 2$. In particular, when $\k\in(0,2)$ the limiting distributions are non-Gaussian with non-diffusive scaling and when $\k = 2$ the limiting distribution is Gaussian but with scaling $\sqrt{n \log n}$. 
However, when $\k>2$ we have the following CLT for both the position $X_n$ of the walk and the hitting times 
\[
 T_n= \inf\{k\geq 0: X_k = n \}, \qquad n\in \Z. 
\]
\begin{thm}[\cite{kksStable,zRWRE}]\label{th:aclt}
 If Assumptions \ref{asm:iid} and \ref{asm:k2} hold, then 
%there exists a constant $\s_0 < \infty$ such that 
\[
 \lim_{n\to\infty} \P\left( \frac{T_n-\frac{n}{\vp}}{\s_0 \sqrt{n}} \leq x \right) 
= \lim_{n\to\infty} \P\left( \frac{X_n-n \vp}{\s_0 \vp^{3/2} \sqrt{n}} \leq x \right)
= \Phi(x), \quad \forall x \in \Z, 
\] 
where 
\[
 \s_0^2 = E[\Var_\w(T_1)] + \Var(E_\w[T_1]) + 2 \sum_{k=1}^\infty \Cov( E_\w[T_1], E_\w^k[T_{k+1}]) < \infty. 
\]
\end{thm}
Theorem \ref{th:aclt} gives CLTs for the RWRE under the averaged measure. However, in this paper we will be primarily interested with CLTs under the quenched measure. 
\begin{thm}[\cite{aRWRE,gQCLT,pThesis}]\label{th:qclt}
 If Assumptions \ref{asm:iid} and \ref{asm:k2} hold, then
\[
 \lim_{n\to\infty} P_\w \left( \frac{T_n-E_\w[T_n]}{\s \sqrt{n}} \leq x \right) 
= \lim_{n\to\infty} P_\w \left( \frac{X_n-n \vp+Z_n(\w)}{\s \vp^{3/2} \sqrt{n}} \leq x \right)
= \Phi(x), \quad P\text{-a.s.} \quad \forall x \in \Z, 
\] 
where 
\[
 \s^2 = E[\Var_\w(T_1)] < \infty \quad \text{and}\quad Z_n(\w) = \vp\left( E_\w[T_{\fl{n\vp}}] - \frac{\fl{n \vp}}{\vp} \right). 
\]
\end{thm}
Before continuing, some important differences between the quenched and averaged CLTs in Theorems \ref{th:aclt} and \ref{th:qclt} should be noted. 
\begin{itemize}
 \item The quenched CLTs in Theorem \ref{th:qclt} require a random (depending on the environment) centering. Indeed, when Assumptions \ref{asm:iid} and \ref{asm:k2} hold it follows from a CLT for sums of ergodic sequences that $\frac{E_\w[T_n]-n/\vp}{\sqrt{n}}$ converges in distribution to a centered Gaussian (see \cite{zRWRE} for details) and therefore one cannot have a quenched CLT for either $T_n$ or $X_n$ with determinstic centering. 
 \item The quenched CLTs are much stronger statements than the averaged CLTs. Indeed, since the quenched probabilities are random variables (randomness coming from the environment $\w$), the limits in the quenched CLTs are required to hold for $P$-a.e.\ environment $\w$. Moreover, the quenched CLTs in Theorem \ref{th:qclt} together with the CLT for $\frac{E_\w[T_n]-n/\vp}{\sqrt{n}}$ can be used to obtain the averaged CLTs in Theorem \ref{th:aclt}. 
 \item Both the quenched and averaged CLTs are known to hold under somewhat more general assumptions than we have used here. In particular, the CLTs have been proved for RWRE in ergodic environments with certain mixing conditions \cite{zRWRE,gQCLT,pThesis}, though in these cases the parameter $\k$ needs to be defined differently than in Assumption \ref{asm:k2} or \eqref{oldkdef}. 
\end{itemize}

\subsection{Main results}

The main results of the present paper concern the rates of convergence in the quenched CLT results in Theorem \ref{th:qclt}. Rates of convergence for the averaged CLT are also of interest, but require different methods and will be studied in a future paper. 

Our approach to the quenched CLTs in this paper will be to follow the approach first used by Alili in \cite{aRWRE} in which one first proves a CLT for the hitting times and then uses this to deduce the CLT for the position of the walk. Therefore, our first two main results concern the rates of convergence for the quenched CLT for hitting times. Note that while the centering in Theorem \ref{th:qclt} needs to be random, the scaling is deterministic. The following two theorems however show that the rate of convergence in the quenched CLT can be improved by using an environment-dependent scaling as well. 

\begin{thm}\label{th:BETn}
Let $\overline{F}_{n,\w}(x) = P_\w\left(  \frac{T_n - E_\w[T_n]}{\sqrt{\Var_\w(T_n)}} \leq x \right)$ be the normalized quenched distribution of $T_n$. 
%Let $\Delta_{n,\w}^T = \sup_x \left| P_\w\left( \frac{T_n - E_\w[T_n]}{\sqrt{\Var_\w(T_n)}} \leq x \right) - \Phi(x) \right|$.
 \begin{itemize}
  \item If $\kappa > 3$, then there exists a constant $C\in(0,\infty)$ such that 
\[
 \limsup_{n\ra\infty}  \sqrt{n} \left\| \overline{F}_{n,\w} - \Phi \right\|_\infty \leq C, \qquad P\text{-a.s.}
\]
 \item If $\k \in (2,3]$, then for any $\e>0$
\[
  \lim_{n\ra\infty} n^{\frac{3}{2}-\frac{3}{\k}-\e} \left\| \overline{F}_{n,\w} - \Phi \right\|_\infty  = 0, \qquad P\text{-a.s.}
\]
 \end{itemize}
\end{thm}

\begin{thm}\label{th:BETnds} 
Let $F_{n,\w}(x) =  P_\w\left(  \frac{T_n - E_\w[T_n]}{\s\sqrt{n}} \leq x \right)$ be the quenched distribution of $T_n$ with random (environment dependent) centering and deterministic scaling. 
\begin{itemize} 
 \item If $\k > 4$, then for any $\e>0$ 
 \[
  \lim_{n\ra\infty} n^{\frac{1}{2}-\e} \left\| F_{n,\w} - \Phi \right\|_\infty  = 0, \quad P\text{-a.s.}
 \]
 \item If $\k \in (2,4]$, then for any $\e>0$ 
 \[
  \lim_{n\ra\infty} n^{1-\frac{2}{\k}-\e} \left\| F_{n,\w} - \Phi \right\|_\infty  = 0, \quad P\text{-a.s.}
 \]
%{\color{red} A unified bound would be 
% \[
%  \lim_{n\ra\infty} n^{1-\frac{2}{4 \wedge \k}-\e} \left\| F_{n,\w} - \Phi \right\|_\infty  = 0, \quad P\text{-a.s.}
% \]
%}
\end{itemize}
\end{thm}

Our final main result is the following bounds on the rates of convergence in the quenched CLT for $X_n$. Note that the results in this theorem give different almost sure and in probability rates of convergence for the quenched CLT. 
 
\begin{thm}\label{th:BEXn}
Let $G_{n,\w}(x) = P_\w\left( \frac{X_n-n\vp +Z_n(\w)}{\s \vp^{3/2} \sqrt{n}} \leq x \right)$. 
 If $\k>2$, then for any $\e>0$
\[
 \lim_{n\to\infty} n^{\frac{1}{4} - \frac{1}{2\k} - \e}\left\| G_{n,\w} - \Phi \right\|_\infty  = 0, \quad P\text{-a.s.}
\]
Moreover, by relaxing the mode of convergence to that of in probability, then the following stronger rates of convergence can be obtained. 
\begin{itemize}
 \item If $\k \in (2,\frac{12}{5})$ then for any $\e>0$, 
\begin{equation}\label{BEXnip1}
  \lim_{n\to\infty} n^{\frac{3}{2} - \frac{3}{\k} - \e}\left\| G_{n,\w} - \Phi \right\|_\infty  = 0, \quad \text{in $P$-probability.}
\end{equation}
 \item If  $\k \geq \frac{12}{5}$ then 
for any $\e>0$, 
\begin{equation}\label{BEXnip2}
  \lim_{n\to\infty} n^{\frac{1}{4} - \e}\left\| G_{n,\w} - \Phi \right\|_\infty  = 0, \quad \text{in $P$-probability.}
\end{equation}
\end{itemize}
\end{thm}

An outline of the proofs of the main results is as follows. 
Sections \ref{sec:Lp} and \ref{sec:asymp} contain analysis of quenched moments of hitting times that will be used later in the proofs of the main results. 
In particular, in Section \ref{sec:Lp} we show that $E[(E_\w[T_1^m])^p] < \infty$ if $p \in (0,\k/m)$, and in Section \ref{sec:asymp} we control the fluctuations of $E_\w[T_n]$ and $\Var_\w(T_n)$ (Section \ref{sec:asymp} is the most technical and difficult part of the paper).  
The proofs of Theorems \ref{th:BETn} and \ref{th:BETnds} are then given in Section \ref{sec:Trates}.
If we let
\[
 \tau_k = T_k - T_{k-1}, \quad k\geq 1, 
\]
then under the quenched measure $P_\w$ the random variables $\{\tau_k\}_{k\geq 1}$ are independent (but not identically distributed). 
Therefore, applying known results for sums of independent random variables gives a bound of $\|\overline{F}_{n,\w}-\Phi\|_\infty$ in terms of the centered quenched moments of the crossing times $\tau_k$. 
Control of these quenched moments then follows from results obtained in Section \ref{sec:Lp} and gives the rates of convergence in Theorem \ref{th:BETn}. 
Since the quenched distributions $\overline{F}_{n,\w}$ and $F_{n,\w}$ differ only in the choice of scaling, Theorem \ref{th:BETnds} then follows from Theorem \ref{th:BETn} and control of the fluctuations of $\Var_\w(T_n) - \s^2 n$ which were obtained in Section \ref{sec:asymp}. 
Finally, in Section \ref{sec:Xrates} the quenched rates of convergence in Theorem \ref{th:BEXn} are obtained from Theorem \ref{th:BETn} in much the same way as the renewal process CLT is obtained from the standard CLT. 
It is here that the need for the quenched centering in the quenched CLTs presents a real difficulty, and in fact the control of the fluctuations of $E_\w[T_n] - n/\vp$ obtained in Section \ref{sec:asymp} are the main contributor to the almost sure rates of convergence in Theorem \ref{th:BEXn}.

\begin{figure}[th]
 \includegraphics[width=0.5\textwidth]{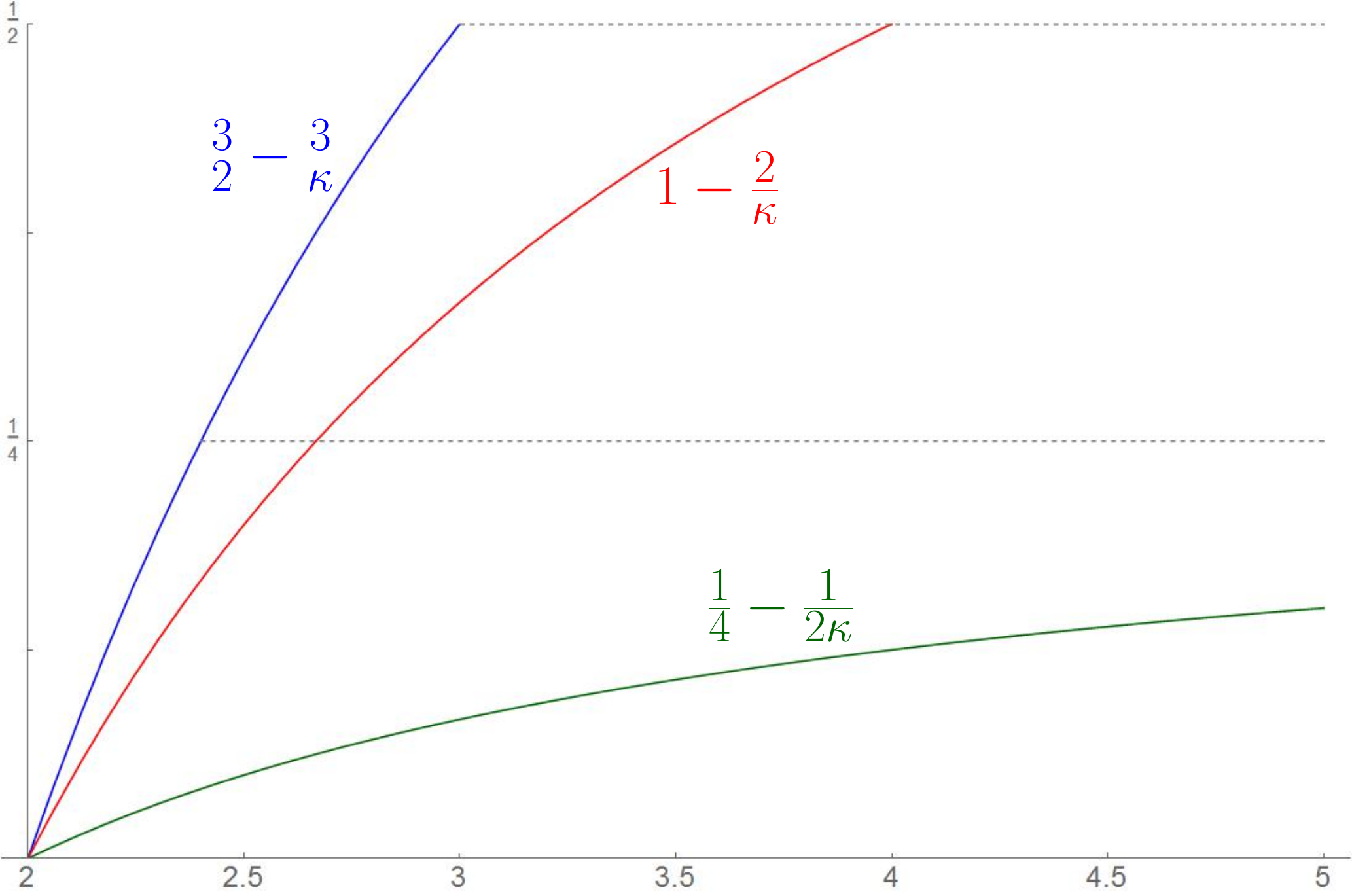}
\caption{A comparison of the different polynomial exponents that appear in Theorems \ref{th:BETn}--\ref{th:BEXn}. The dotted lines are at height $1/4$ and $1/2$.}
\end{figure}

\subsection{Discussion of main results and future work}

Central limit theorems for random motion in random media are closely related to problems in stochastic homogenezation, a connection going back at least to  Papanicolaou and Varadan \cite{pvDWRC}. 
Results in quantitative stochastic homogenization (that is, results which give bounds on the rate of convergence of the solution of a PDE with random coefficients to the solution of the deterministic homogenized PDE) were first obtained by Yurinski\u\i \cite{yASONDE,yEAMD}, but recently there have been a number of important breakthroughs \cite{csROCH,goOVESH,asQSHEENDF,gnoQOESH}.
However, the only results of which we are aware of giving quantitative rates of convergence for central limit theorems for RWRE are by Mourrat in \cite{mQCLTRWRC}. Mourrat's results differ from those in the present paper as they are for the random conductance model of RWRE rather than for RWRE in i.i.d.\ and he proves quantitative rates of convergence for the averaged CLT rather than the quenched CLT. Mourrat also gives rates of convergence for the random conductance model in any dimension $d\geq 1$, while our methods are restricted to one dimension.
It should also be noted that the martingale method that Mourrat uses is limited to proving at best rates of convergence of $n^{-1/5}$, while the rates of convergence in Theorems \ref{th:BETn}--\ref{th:BEXn} are in many cases faster than $n^{-1/5}$.

A natural question regarding the main results of this paper is the optimality of the rates of convergence obtained. The quenched rates in Theorem \ref{th:BETn} when $\k>3$ are clearly optimal, though it is not clear if the other rates in Theorems \ref{th:BETn} and \ref{th:BETnds} are optimal. However, we conjecture that they are optimal in the sense that no better almost sure polynomial rate can be obtained. In particular, if one sets $\e=0$ in these results we conjecture that the limits do not exist. For instance,  if $\k\in (2,3)$ we conjecture that 
\[
 \liminf_{n\to\infty} n^{\frac{3}{2}-\frac{3}{\k}} \| \overline{F}_{n,\w} - \Phi \|_\infty = 0, \quad\text{and}\quad \limsup_{n\to\infty} n^{\frac{3}{2}-\frac{3}{\k}} \| \overline{F}_{n,\w} - \Phi \|_\infty = \infty, \quad P\text{-a.s.}
\]

It is less clear to us if the rates of convergence in Theorems \ref{th:BEXn} are optimal or not. In particular, one wonders if a different method of proof of the quenched CLT for $X_n$ would lead to an improved rate of convergence. Zeitouni outlines in \cite{zRWRE} how the quenched CLT for $X_n$ can be obtained from a martingale CLT via the ``harmonic corrector'' approach. However, since the fluctuations of the harmonic corrector are given in this case by the fluctuations of $E_\w[T_n]-n/\vp$ it seems that this approach will not yield any better results than that of the approach in the current paper. 
Also, given that Theorems \ref{th:BETn} and \ref{th:BETnds} show that the choice of normalization can affect the rates of convergence, one wonders if the rates of convergence in Theorem \ref{th:BEXn} can be improved by using a different centering or an environment-dependent scaling. In particular, one might suspect that better rates of convergence can be obtained for $\frac{X_n-E_\w[X_n]}{\s \vp^{3/2} \sqrt{n}}$ or $\frac{X_n-E_\w[X_n]}{\sqrt{\Var_\w(X_n)}}$. Unfortunately, as of now we are not aware of any proofs of the quenched central limit theorem that work for these normalizations directly. (Of course one might be able to obtain a quenched CLT for $\frac{X_n-E_\w[X_n]}{\s \vp^{3/2} \sqrt{n}}$ indirectly by using the quenched CLT for $\frac{X_n-n\vp+Z_n(\w)}{\s \vp^{3/2} \sqrt{n}}$ and then proving that $\frac{E_\w[X_n]-n\vp+Z_n(\w)}{\sqrt{n}} \to 0$, $P$-a.s., but this would not lead to any possible improvement in the rate of convergence for the quenched CLT.)

\section{Quenched moments of hitting times}\label{sec:Lp}

In this section we will collect some facts about quenched moments of hitting times that will be useful later.
The main result is the following $L^p$ estimate for the quenched moments of hitting times. 
 
%Many of these facts were known previously, though sometimes under slightly weaker assumptions than we are using in the present paper. 

\begin{lem}\label{lem:taumtail}
 If $E[\log \rho_0] < 0$ and $\k > 0$, then for any integer $m\geq 1$, $E\left[E_\w[\tau_1^m]^p \right] < \infty$ for all $p \in (0,\frac{\k}{m})$. 
\end{lem}
\begin{rem}
 We will only need Lemma \ref{lem:taumtail} for $m\leq 3$ and $\k > 2$ in the present paper. Nevertheless, since the proof generalizes easily to all $m\geq 1$ we give the more general proof here. 
\end{rem}
\begin{rem}
If the parameter $\k$ satisfies the slightly stronger definition \eqref{oldkdef}, and if the technical conditions $E[\rho_0^\k \log \rho_0] < \infty$ and the distribution of $\log \rho_0$ is non-lattice are also satisfied, then it is known that $P(E_\w[\tau_1] > x ) \sim C x^{-\k}$ as $x\to\infty$ which is a stronger statement than the $L^p$ bounds in the statement of Lemma \ref{lem:taumtail} for the case $m=1$.  We conjecture that under these stronger assumptions that similar tail asymptotics hold for $E_\w[\tau_1^m]$ also; that is, we conjecture that for any $m\geq 1$ there exists a constant $C_m>0$ such that $P( E_\w[\tau_1^m] > x) \sim C_m x^{-\k/m}$ as $x\to\infty$. 
However, since such precise tail asymptotics are not needed for our purposes in this paper we content ourselves with the weaker $L^p$ bounds given here. 
\end{rem}

\begin{proof}
We begin by computing recursive formulas for $E_\w[\tau_1^m]$.
To this end it is helpful to introduce the natural left shift operator $\th$ on the space of environments. That is, for any $k \in \Z$, $\th^k \w$ is the environment with $(\th^k\w)_x = \w_{x+k}$ for every $x\in \Z$. 
With this notation, by conditioning on the first step of the walk, 
\begin{align*}
 E_\w[\tau_1^m] &= \w_0 + (1-\w_0)E_\w^{-1}[(1+T_1)^m] \\
&= \w_0 + (1-\w_0) E_{\th^{-1}\w}\left[ \left(1+\tau_1+\tau_2\right)^m \right] \\
%&= \w_0 + (1-\w_0) \sum_{\substack{k_1,k_2\geq 0 \\ k_1+k_2\leq m}} \binom{m}{k_1,k_2,m-k_1-k_2} E_{\th^{-1}\w}[\tau_1^{k_1}] E_\w[\tau_1^{k_2}] \\
&= \w_0 + (1-\w_0) \sum_{\substack{0\leq k_1,k_2 < m \\ k_1+k_2\leq m}} \binom{m}{k_1,k_2,m-k_1-k_2} E_{\th^{-1}\w}[\tau_1^{k_1}] E_\w[\tau_1^{k_2}] \\
&\qquad  + (1-\w_0) E_{\th^{-1}\w}[\tau_1^m] + (1-\w_0) E_\w[\tau_1^m]. 
\end{align*}
Assuming for the moment that all of the above quenched expectations are finite we can solve this for $E_\w[\tau_1^m]$ to obtain 
\begin{align}
  E_\w[\tau_1^m] &= 1 + \rho_0 \sum_{\substack{0\leq k_1,k_2 < m \\ k_1+k_2\leq m}} \binom{m}{k_1,k_2,m-k_1-k_2} E_{\th^{-1}\w}[\tau_1^{k_1}] E_\w[\tau_1^{k_2}] + \rho_0 E_{\th^{-1}\w}[\tau_1^m] \label{EwT1mrec}\\
&=: f_m(\w) + \rho_0 E_{\th^{-1}\w}[\tau_1^m], \nonumber 
\end{align}
(where the last equality gives the definition of $f_m(\w)$), and iterating this we obtain 
%\[
%E_\w[\tau_1^m]
%= f_m(\w) + \rho_0 f_m(\th^{-1}\w) + \rho_0 \rho_{-1} f_m(\th^{-2}\w) + \cdots + \rho_0\rho_{-1}\cdots\rho_{-k+1} f_m(\th^{-k}\w) + \rho_0\rho_{-1}\cdots \rho_{-k} E_{\th^{-k-1}\w}[\tau_1^m]
%\]
\[
 E_\w[\tau_1^m] = f_m(\w) + \sum_{k=1}^{n-1} \Pi_{-k+1,0} f_m(\th^{-k}\w) + \Pi_{-n+1,0} E_{\th^{-n}\w}[\tau_1^m], \qquad\text{where } \Pi_{i,j} = \prod_{x=i}^j \rho_x \text{ for any } i \leq j. 
\]
In the argument thus far we have been assuming that all the quenched expectations are finite which may not necessarily be true. 
To account for this we can modify the environment by adding a reflection to the right at a point to the left of the origin. In particular, for any $n\geq 1$ let $\w(n) = \{\w(n)_x\}_{x \in \Z}$ be the environment such with a reflection added at $x=-n$. That is,  
\[
 \w(n)_x = \begin{cases} \w_x & \text{if } x \neq -n \\ 1 & \text{if } x = -n. \end{cases}
\]
The added reflection makes it so that $\tau_1$ has exponential tails under the measure $P_{\w(n)}$ so that in particular $E_{\th^x\w(n)}[\tau_1^m] < \infty$ for any $x \geq -n$. Therefore, repeating the above recursive argument in the environment $\w(n)$ gives 
\[
 E_{\w(n)}[\tau_1^m] = f_m(\w(n)) + \sum_{k=1}^{n-1} \Pi_{-k+1,0} f_m(\th^{-k}\w(n)) + \Pi_{-n+1,0}.
\]
We wish to then take $n\to\infty$ in the above to obtain a formula for $E_\w[\tau_1^m]$. Since $E[\log\rho_0]< 0$ the last term on the right vanishes almost surely as $n\to \infty$. For the other terms, by coupling the path of the walk in the environment $\w$ to the paths in $\w(n)$ up to the stopping time $T_{-n}$ we see that
$E_{\th^x\w(n)}[\tau_1^\ell] \nearrow E_{\th^x\w}[\tau_1^\ell]$ as $n\nearrow\infty$ for any fixed $x$ and $\ell$. In particular, this implies that $f_m(\th^{-k}\w(n))$ is non-decreasing in $n$ and so the monotone convergence theorem implies that 
\begin{equation}\label{EwT1mform}
 E_\w[\tau_1^m] = f_m(\w) +  \sum_{k=1}^\infty \Pi_{-k+1,0} f_m(\th^{-k}\w).
\end{equation}

We will now use \eqref{EwT1mform} to prove the moment bounds for $E_\w[\tau_1^m]$. 
A key tool that we will use in the proof is the following simple lemma which follows from Minkowski's inequality when $p\geq 1$ and the sub-additivity of $x\mapsto x^p$ when $p \in (0,1)$. 
\begin{lem}\label{lem:Lpsum}
 Let $Y_1,Y_2,\ldots$ be non-negative random variables and let $Z = \sum_{k=0}^\infty Y_k$. 
\begin{itemize}
 \item If $p < 1$ and $\sum_{k=0}^{\infty} E[Y_k^p] < \infty$ then $E[Z^p] <\infty$. 
 \item If $p \geq 1$ and $\sum_{k=0}^{\infty} E[Y_k^p]^{1/p} < \infty$ then $E[Z^p] <\infty$.
\end{itemize}
\end{lem}
\noindent
By this lemma and \eqref{EwT1mform} it will be enough to show that $E\left[(\Pi_{-k+1,0} f_m(\th^{-k}\w) )^p \right]$ is decreasing exponentially fast if $p \in (0,\frac{\k}{m})$. 
To prove this, first note $f_m(\w)$ depends only on the environment to the left of the origin. Therefore, since the environment is i.i.d.\ we have that 
\[
 E\left[(\Pi_{-k+1,0} f_m(\th^{-k}\w) )^p \right] = E\left[(\Pi_{-k+1,0})^p\right] E\left[ f_m(\w)^p \right] = (r_p)^k E\left[ f_m(\w)^p \right].
\]
Since it follows from \eqref{rpdef} that $r_p < 1$ we have thus reduced ourselves to proving 
%Since $r_p = E[\rho_0^p] <1$ if $p \in (0,\k)$, we have thus reduced ourselves to proving 
\begin{equation}\label{fmLp}
 E[f_m(\w)^p] < \infty, \quad \text{for all } p \in \left(0,\frac{\k}{m}\right), \, m \geq 1. 
\end{equation}
We will prove \eqref{fmLp} by induction on $m\geq 1$. 
In the case $m=1$ we have that $f_1(\w) = 1 + \rho_0$ and so $E[f_1(\w)^p] = E[(1+\rho_0)^p] < \infty$ for $p \in (0,\k)$ holds. 
Next, we will assume that \eqref{fmLp} holds up to $m-1$; that is, we will assume that $E[E_\w[\tau_1^k]^p] < \infty$ for any $p\in (0,\frac{\k}{k})$ and $k\leq m-1$. 
Under this assumption, if $0\leq k_1,k_2 < m$ and $k_1+k_2 \leq m$ then H\"older's inequality implies that 
\[
 E\left[ \left( E_{\th^{-1}\w}[\tau_1^{k_1}] E_\w[\tau_1^{k_2}] \right)^p \right] 
\leq \left( E\left[ E_\w[\tau_1^{k_1}]^{\frac{mp}{k_1}}  \right] \right)^{\frac{k_1}{m}} \left( E\left[ E_\w[\tau_1^{k_2}]^{\frac{mp}{m-k_1}}  \right] \right)^{\frac{m-k_1}{m}} < \infty \quad \text{if } p \in \left(0,\frac{\k}{m}\right), 
\]
where the expectations on the right are finite by the induction assumption since $\frac{mp}{k_1} < \frac{\k}{k_1}$ and $\frac{mp}{m-k_1} \leq \frac{mp}{k_2} < \frac{\k}{k_2}$. 
This is enough to conclude that \eqref{fmLp} holds for $m$ as well, and by induction for all $m\geq 1$. 
\end{proof}

We close this section by noting some additional consequences of the recursive formula for $E_\w[\tau_1^m]$ that will be useful later in the paper.
For ease of notation we will introduce the following notation for the quenched mean and variance of hitting times that will be used throughout the paper. 
\[
 \mu_k = E_{\th^k\w}[\tau_1] \quad \text{and} \quad V_k = \Var_{\th^k\w}(\tau_1). 
\]
When $m=1,2$, the recursive formula \eqref{EwT1mrec} (applied to the shifted environment $\th^k\w$) yields  
%\[
% E_\w[\tau_1] = 1 + \rho_0 + \rho_0 E_{\th^{-1}\w}[\tau_1],
%\]
\begin{equation}\label{mukrec}
 \mu_k = 1 + \rho_k + \rho_k \mu_{k-1},
\end{equation}
and
%\[
% E_\w[\tau_1^2] = 1 + \rho_0\left( 1 + 3 E_{\th^{-1}\w}[\tau_1] + 3E_\w[\tau_1] + 6 E_{\th^{-1}\w}[\tau_1]E_\w[\tau_1] \right) + \rho_0 E_{\th^{-1}\w}[\tau_1^2]. 
%\]
\[
 E_{\th^k\w}[\tau_1^2] = 1 + \rho_k\left( 1 + 2 \mu_{k-1} + 2 \mu_k + 2 \mu_{k-1}\mu_k + E_{\th^{k-1}\w}[\tau_1^2]\right).
\]
Inserting the first formula into the second and then simplifying yields 
\begin{align}
 E_{\th^k\w}[\tau_1^2] 
%&= 1 + \rho_k\left( 1 + 2 \mu_{k-1} + 2 \mu_k + 2 \mu_{k-1}\mu_k + E_{\th^{k-1}\w}[\tau_1^2] \right) \nonumber\\
%&= 1 + \rho_k\left( 1 + 2 \mu_{k-1} + 2 \mu_k(1+\mu_{k-1}) + E_{\th^{k-1}\w}[\tau_1^2] \right)\nonumber\\
%&= 1 + \rho_k\left( 1 + 2 \mu_{k-1} + 2 (1 + \rho_k + \rho_k \mu_{k-1})(1+\mu_{k-1}) + E_{\th^{k-1}\w}[\tau_1^2] \right) \nonumber \\
&= 1 + \rho_k\left( 1 + 2 \mu_{k-1} + 2 \left\{(1 + \rho_k) + (1+2\rho_k)\mu_{k-1} + \rho_k\mu_{k-1}^2 \right\} + E_{\th^{k-1}\w}[\tau_1^2] \right) \nonumber \\
&= 1 + \rho_k\left( 1 + 2(1+\rho_k) + 4(1+\rho_k)\mu_{k-1} + 2\rho_k \mu_{k-1}^2 +   E_{\th^{k-1}\w}[\tau_1^2] \right) \nonumber \\
&= (1+\rho_k)(1+2\rho_k) + 4\rho_k(1+\rho_k) \mu_{k-1} + 2 \rho_k^2 \mu_{k-1}^2 + \rho_k E_{\th^{k-1}\w}[\tau_1^2]. \label{EwT12rec}
\end{align}
Combining \eqref{mukrec} and \eqref{EwT12rec} then yields the following recursive formula for the quenched variance. 
\begin{align}
 V_k = E_{\th^k\w}[\tau_1^2] - \mu_k^2 
&= (1+\rho_k)(1+2\rho_k) + 4\rho_k(1+\rho_k) \mu_{k-1} + 2 \rho_k^2 \mu_{k-1}^2 + \rho_k E_{\th^{k-1}\w}[\tau_1^2] \nonumber \\
&\qquad - (1+\rho_k)^2-2\rho_k(1+\rho_k)\mu_{k-1} - \rho_k^2 \mu_{k-1}^2 \nonumber \\
&= \rho_k(1+\rho_k) + 2\rho_k(1+\rho_k)\mu_{k-1} + \rho_k^2 \mu_{k-1}^2 + \rho_k  E_{\th^{k-1}\w}[\tau_1^2] \nonumber \\
%&= \rho_k(1+\rho_k) + 2\rho_k(1+\rho_k)\mu_{k-1} + \rho_k^2 \mu_{k-1}^2 + \rho_k\mu_{k-1}^2 + \rho_k  V_{k-1}  \nonumber \\
&= \rho_k(1+\rho_k) + 2\rho_k(1+\rho_k)\mu_{k-1} + \rho_k(1+\rho_k) \mu_{k-1}^2 + \rho_k  V_{k-1} \nonumber \\
%&= \rho_k(1+\rho_k)\left( 1 + 2 \mu_{k-1} + \mu_{k-1}^2 \right) + \rho_k V_{k-1}  \nonumber \\
&= \rho_k(1+\rho_k)(1+\mu_{k-1})^2 +  \rho_k V_{k-1}. \label{Vkrec}
\end{align}

Finally, we note that since $\mu_{k-1}$ is independent of $\rho_k$, one can take expectations of both sides of \eqref{mukrec} (or square both sides and then take expectations) to obtain the following explicit formulas for the first two moments of the quenched hitting times. 
\begin{equation}\label{mumoments}
\frac{1}{\vp} = \E[\tau_1] = E[\mu_0] = \frac{1+r_1}{1-r_1}, \qquad  E[\mu_0^2] = \frac{1+3r_1+3r_2+r_1r_2}{(1-r_1)(1-r_2)}. 
\end{equation}
This formula for $\E[\tau_1]$ is well known and in fact was originally obtained in this manner in Solomon's seminal paper \cite{sRWRE}. 
Similarly, taking expectations of both sides of \eqref{Vkrec} and using the formulas in \eqref{mumoments} and the fact that $V_{k-1}$ is independent of $\rho_k$ one can obtain
\begin{equation}\label{Vmean}
 \s^2 = E[\Var_\w(\tau_1)] = E[V_0] = \frac{4 (1+r_1)(r_1+r_2)}{(1-r_2)(1-r_1)^2}.
\end{equation}
We will briefly provide the details of this argument for \eqref{Vmean} since the formula here corrects for a small typo in the formula given in \cite{gQCLT}. 
%Taking expectations in \eqref{Vkrec} and using the formulas in \eqref{mumoments} one obtains 
\begin{align*}
 E[V_0] &= (r_1+r_2)\left( 1 + 2E[\mu_0] + E[\mu_0^2]\right) + r_1 E[V_0] \\
&= (r_1+r_2)\left( 1 + 2\frac{1+r_1}{1-r_1} +  \frac{1+3r_1+3r_2+r_1r_2}{(1-r_1)(1-r_2)}\right) + r_1 E[V_0] \\
&= (r_1+r_2)\left( \frac{4(1+r_1)}{(1-r_1)(1-r_2)} \right) + r_1 E[V_0]. 
\end{align*}
Solving this for $E[V_0]$ we obtain the formula in \eqref{Vmean}. 
%We omit the details of the proofs of \eqref{mumoments} and \eqref{Vmean} since the formulas for $E[\mu_0]$ and $E[V_0]$ were given previously in \cite{sRWRE} and \cite{gQCLT}, respectively (note however that the formula for $E[V_0]$ in \cite{gQCLT} contains a small typo which is corrected here), and . 

\section{Asymptotics of the quenched mean and variance of the hitting times}\label{sec:asymp}

Since $E_\w[T_n] = \sum_{k=0}^{n-1} \mu_k$ and $\Var_\w(T_n) = \sum_{k=0}^{n-1} V_k$ and since $\{\mu_k\}_{k\in\Z}$ and $\{V_k\}_{k\in\Z}$ are ergodic sequences, it follows that $E_\w[T_n]/n \to E[\mu_0] = \E[\tau_1] = \frac{1}{\vp}$ and $\Var_\w(T_n)/n \to E[V_0] = \s^2$, almost surely as $n\to \infty$. However, for the proofs of our main results we will need control on the fluctuations of $E_\w[T_n]$ and $\Var_\w(T_n)$ from these deterministic limits. 
The first such result we need is the following lemma which was proved by Goldsheid. 

\begin{lem}[Lemma 4 in \cite{gQCLT}]\label{lem:qETmd}
 If $\k>2$, then for any $\e>0$, 
\[
 \lim_{n\to\infty} \frac{E_\w[T_n] - \frac{n}{\vp}}{n^{1/2+\e}} = 0, \quad P\text{-a.s.}
\]
\end{lem}

The main results of this section are the following two propositions, the first of which gives an improvement to Lemma \ref{lem:qETmd} by controlling the fluctuations the quenched mean of hitting times of nearby locations and the second of which 
which controls the fluctuations of the quenched variance of hitting times.

%Note that Lemma \ref{lem:qETmd} is not surprising since one can use an CLT for sums of ergodic sequences to prove that in fact $\{\frac{E_\w[T_{\fl{nt}}]-\fl{nt}/\vp}{\sqrt{n}} \}_{t\geq 0}$ converges in distribution to a zero mean Brownian motion (we will not use this fact here, but see Section 2.2 in \cite{zRWRE} for the proof of this). 

\begin{prop}\label{prop:qETfmd}
For any $n\geq 1$ and $\e>0$ denote $I_{\e,n} = [n \vp-n^{1/2+\e},n\vp+n^{1/2+\e}]$. 
If $\k>2$, then 
\begin{equation}\label{fmdip}
 \lim_{n\ra\infty} \max_{k,\ell \in I_{\e,n}} \frac{ \left| E_\w[T_k] - E_\w[T_\ell]-\frac{k-\ell}{v} \right| }{n^{1/4+\e/2+\e'}} = 0, \quad \text{in $P$-probability,}\quad \text{for any } \e'>0,
\end{equation}
and 
\begin{equation}\label{fmdas}
 \lim_{n\ra\infty} \max_{k,\ell \in I_{\e,n}} \frac{ \left| E_\w[T_k] - E_\w[T_\ell]-\frac{k-\ell}{v} \right| }{n^{\frac{1}{4} + \frac{1}{2\k} + \e(\frac{1}{2}-\frac{1}{\k})+\e'}} = 0, \quad P\text{-a.s.,}\quad \text{for any } \e'>0.
\end{equation}
\end{prop}

\begin{prop}\label{prop:qVmd1}
% If $\k>4$ then for any $\a>0$, 
%\[
% \lim_{n\ra\infty} \frac{\Var_\w(T_n) - \s^2 n }{n^{1/2+\a}} = 0, \quad P\text{-a.s.}
%\]
% If $\k \in (2,4)$ then for any $\a>0$, 
%\[
% \lim_{n\ra\infty} \frac{\Var_\w(T_n) - \s^2 n }{n^{2/\k+\a}} = 0, \quad P\text{-a.s.}
%\]
If $\k > 2$, then for any $\e>0$, 
\[
 \lim_{n\ra\infty} \frac{\Var_\w(T_n) - \s^2 n}{n^{\frac{2}{4 \wedge\k} + \e}} = 0, \quad P\text{-a.s.}
\]
\end{prop}

\begin{rem}
 It was shown in \cite[Lemma 5]{gQCLT} that for any $\e>0$
\[
 \lim_{n\ra\infty} \max_{k,\ell \in I_{\e,n}} \frac{ \left| E_\w[T_k] - E_\w[T_\ell]-\frac{k-\ell}{v} \right| }{\sqrt{n}} = 0, \quad P\text{-a.s.}
\]
Thus, for $\e \in (0,1/2)$, Proposition \ref{prop:qETfmd} is an improvement on the results in \cite{gQCLT}. 
\end{rem}

\begin{rem}
 The change in the magnitude of the fluctuations of $\Var_\w(T_n)$ at $\k=4$ in Proposition \ref{prop:qVmd1} is due to the fact that $\Var_\w(T_n)$ has finite second moment when $\k>4$. In fact, though we will not need this here, it can be shown that if $\k>4$ then $\frac{\Var_\w(T_n)-\s^2 n}{\sqrt{n}}$ converges in distribution to a zero mean Gaussian random variable. We also suspect that under additional regularity assumptions ($E[\rho_0^\k]=1$,  $E[\rho_0^\k \log \rho_0]<\infty$ and the distribution of $\log \rho_0$ is non-lattice) that if $\k \in (2,4)$ then $\frac{\Var_\w(T_n)-\s^2 n}{n^{2/\k}}$ converges in distribution to a $\k/2$-stable random variable. 
\end{rem}

The main idea of the proofs of both Propositions \ref{prop:qETfmd} and \ref{prop:qVmd1} is that $E_\w[T_n]-n/\vp$ and $\Var_\w(T_n)-\s^2 n$ can be approximated by martingales which are sums of stationary ergodic sequences. 
To this end, it will be helpful to first state and prove the following general lemma. 
\begin{lem}\label{lem:Mnmd}
 Let that $\{Z_k\}_{k\in\Z}$ be an ergodic sequence and let $\{W_n\}_{n\geq 0}$ be the martingale defined by $W_0 = 0$ and 
\[
 W_n = \sum_{k=0}^{n-1} \left( Z_k - E[Z_k | \mathcal{F}_{k-1}] \right), \quad \text{where } \mathcal{F}_k = \s( Z_j : \, j\leq k ), 
\]
and let $W_n^* = \max_{k\leq n} |W_k|$. 
If $E[|Z_1|^p] < \infty$ for all $1\leq p < \a$, then for any $\e>0$,
\[
 \lim_{n\ra\infty} \frac{ W_n^* }{n^{\frac{1}{\a \wedge 2} + \e}} = 0, \quad P\text{-a.s.}
\]
Additionally, if $\a>2$ then $E[|W_n^*|^p] = \bigo(n^{p/2})$ for all $p \in[2,\a)$. 

\end{lem}

\begin{proof}
We will divide the proof into two cases: $\a>2$ and $\a \in (1,2]$.
In both cases, however we will use that 
\[
 E\left[ \sum_{k=1}^n |W_k - W_{k-1}|^p \right] = \sum_{k=1}^n E[|Z_{k-1} - E[Z_{k-1} \, | \, \mathcal{F}_{k-2}] |^p] = n E[|Z_1-E[Z_1 \, | \, \mathcal{F}_0] |^p] = \bigo(n), \quad \text{for } p<\a. 
\]

\textbf{Case I: $\a>2$.} If $p \in [2,\a)$, it follows from the Burkholder-Davis-Gundy inequality and then Jensen's inequality that there exists a constant $C_p>0$ depending only on $p$ such that 
\begin{align*}
 E[|W_n^*|^p] &\leq C_p E\left[ \left( \sum_{k=1}^n (W_k-W_{k-1})^2 \right)^{p/2} \right] \\
&\leq C_p n^{p/2-1} E \left[ \sum_{k=1}^n|W_k-W_{k-1}|^p \right] = \bigo(n^{p/2}). 
\end{align*}
From this it follows that $P(W_n^* > \d n^{1/2+\e}) = \bigo( n^{-\e p} )$, and so if we let $n_k = \ceil{k^{2/(\e p)}}$ it follows from the Borel-Cantelli Lemma that 
\[
 \lim_{k\ra\infty} \frac{W_{n_k}^*}{n_k^{1/2+\e}} = 0, \qquad P\text{-a.s.}
\]
Finally, since $W_n^*$ is non-decreasing in $n$ and $n_{k+1}/n_k \to 1$ as $k\to \infty$ the conclusion of the lemma follows easily. 

\textbf{Case II: $\a \in (1,2]$.} If $p \in [1,\a)$ then the Burkholder-Davis-Gundy inequality implies that 
\begin{align*}
 E[|W_n^*|^p] &\leq  C_p E\left[ \left( \sum_{k=1}^n (W_k-W_{k-1})^2 \right)^{p/2} \right] \\
&\leq C_p E\left[ \sum_{k=1}^n |W_k-W_{k-1}|^p \right] = \bigo(n), 
\end{align*}
where in the second inequality we used that $p/2 < 1$. 
Therefore, if 
%we choose $p$ such that
$\max \{1, \frac{1+\e\a/2}{1/\a+\e} \} \leq p < \a$, then 
\[
 P\left( W_n^* > \d n^{1/\a+\e} \right) = \bigo\left( n^{1-p(\frac{1}{\a}+\e)} \right) = \bigo( n^{-\e\a/2}), 
\]
where the last equality follows from $1-p(\frac{1}{\a}+\e) \leq 1-\frac{1+\e\a/2}{1/\a+\e}(\frac{1}{\a}+\e) = -\frac{\e\a}{2}$. 
Letting $n_k = \ceil{ k^{4/(\e\a)} }$, it follows from the Borel-Cantelli Lemma that 
\[
 \lim_{k\ra\infty} \frac{W_{n_k}^*}{n_k^{1/\a+\e}} = 0, \quad P\text{-a.s.}
\]
As in Case I, the conclusion of the lemma follows easily from this since $n_{k+1}/n_k \to 1 $ as $k\to\infty$. 
\end{proof}

We are now ready to give the proofs of the main results of this section. 

\begin{proof}[Proof of Proposition \ref{prop:qETfmd}]
Consider the martingale defined by $M_0=0$ and 
\[
 M_n = \sum_{k=0}^{n-1} \left( \mu_k - E[\mu_k \, | \, \mathcal{F}_{k-1} ] \right), \quad n\geq 1, 
\qquad \text{where } \mathcal{F}_m = \s(\w_x: \, x \leq m). 
\]
To see the relevance of this martingale, note that 
it follows from the recursion for $\mu_k$ in \eqref{mukrec} and the fact that $\rho_k$ is independent of $\mathcal{F}_{k-1}$ that $E[\mu_k\, | \, \mathcal{F}_{k-1} ] = 1+r_1 + r_1 \mu_{k-1}$.
%since $\mu_k = 1+\rho_k + \rho_k \mu_{k-1}$ then since $\rho_k$ is independent of $\mathcal{F}_{k-1}$ we have $E[\mu_k\, | \, \mathcal{F}_{k-1} ] = 1+r_1 + r_1 \mu_{k-1}$. 
Using this we can re-write the martingale as 
\begin{align*}
 M_n
%= \sum_{k=0}^{n-1} \left(\mu_k - 1 - r_1 - r_1 \mu_{k-1} \right) 
&= \sum_{k=0}^{n-1} \mu_k - (1+r_1)n - r_1 \sum_{k=-1}^{n-2} \mu_k \\
&= (1-r_1) \sum_{k=0}^{n-1} \mu_k - (1+r_1) n + r_1( \mu_{n-1} - \mu_{-1} ) \\
&= (1-r_1) \left( E_\w[T_n] - \frac{n}{\vp} \right) + r_1( \mu_{n-1} - \mu_{-1} ),
\end{align*}
where in the last equality we used the explicit formula for $\vp$ in \eqref{mumoments}. 
It follows from this representation of the martingale that
\begin{equation}\label{fmd2parts}
 \max_{k,\ell \in I_{\e,n}} \left| E_\w[T_k] - E_\w[T_\ell]-\frac{k-\ell}{v} \right| \leq  \max_{k,\ell \in I_{\e,n}}  \frac{|M_k-M_\ell|}{1-r_1} + \frac{2r_1}{1-r_1} \max_{k\in I_{\e,n}} \mu_{k-1}. 
\end{equation}
%Therefore, we can control the differences in the quenched expectations by controlling the two terms on the right in \eqref{fmd2parts}. 
%We will use \eqref{fmd2parts} to prove both the convergence in probability and almost sure convergence statements in the statement of Proposition \ref{prop:qETfmd}. 
To control the first term on the right in \eqref{fmd2parts}, it follows from Lemmas \ref{lem:taumtail} and \ref{lem:Mnmd} imply that for any $p \in [2,\k)$ there exists a constant $C>0$ such that 
\[
 E\left[ \max_{\ell \in [k,k+n]} |M_\ell - M_k|^p \right] \leq C n^{p/2}, \qquad \forall k\geq 0, 
\]
and thus
\begin{align}
 P\left(\max_{k,\ell \in I_{\e,n}} |M_k-M_\ell| \geq \d n^{1/4+\e/2+\e'} \right)
&\leq P\left(\max_{\ell \in I_{\e,n}} |M_\ell-M_{\ceil{n\vp-n^{1/2+\e}}}| \geq \frac{\d}{2} n^{1/4+\e/2+\e'} \right)  \nonumber \\
&\leq \frac{E\left[ \max_{\ell \in I_{\e,n}} |M_\ell-M_{\ceil{n\vp-n^{1/2+\e}}}|^p   \right]}{(\d/2)^p n^{p(1/4+\e/2+\e')}} \nonumber \\
%%%%%&\leq \frac{ C n^{(p/2)(1/2+\e)}}{(\d/2)^p n^{p(1/4+\e/2+\e')}} \nonumber \\
&= \bigo\left( n^{-p\e'}  \right). \label{Mfluc}
\end{align}
To control the second term on the right in \eqref{fmd2parts}, note that it follows from Lemma \ref{lem:taumtail} and a $p$-th moment bound for $p \in[2,\k)$ that 
\begin{equation}\label{mukfluc}
 P\left( \max_{k\in I_{\e,n}} \mu_{k-1} > \d n^{\frac{1}{4}+\frac{\e}{2} + \e'} \right) 
\leq | I_{\e,n}| P\left(\mu_0 > \d n^{\frac{1}{4}+\frac{\e}{2} + \e'} \right)
= \bigo\left( n^{\frac{1}{2}+\e - p\left( \frac{1}{4} +\frac{\e}{2} + \e'\right) } \right)
= \bigo\left( n^{-p\e'} \right). 
\end{equation}
Applying \eqref{Mfluc} and \eqref{mukfluc} to \eqref{fmd2parts} proves the convergence in probability statement in \eqref{fmdip}. 

For the proof of the almost sure convergence in \eqref{fmdas} we will use the bounds in \eqref{Mfluc} and \eqref{mukfluc} but we will need to restrict ourselves to $\e'> \frac{1/2-\e}{\k}$. 
For any such $\e'$, fix $p$ such that $\max\{2,\frac{1/2-\e}{\e'} \} < p < \k$ and then  $\gamma>0$ such that $\frac{1}{2}-\e < \frac{1}{\gamma} < p \e'$.  
%If we let $n_k = \fl{k^\gamma}$ then since $\gamma p \e' > 1$ it follows from \eqref{Mfluc} and \eqref{mukfluc} that 
%\begin{equation}\label{Mnfmdnk}
%\lim_{k\ra\infty} \max_{\ell,m \in I_{\e,n_k}} \frac{|M_m-M_\ell|}{n_k^{1/4+\e/2+\e'}} = 
% \lim_{k\ra\infty} \max_{\ell \in I_{\e,n_k}} \frac{\mu_{\ell-1}}{n_k^{1/4+\e/2+\e'}} = 0, \quad P\text{-a.s.}
%%, \quad \text{for } \e'>\frac{1/2-\e}{\k}. 
%\end{equation}
If we let $n_k = \fl{k^\gamma}$ then since $\gamma p \e' > 1$ it follows from \eqref{Mfluc} and \eqref{mukfluc} applied to \eqref{fmd2parts} that 
\begin{equation}\label{Mnfmdnk}
\lim_{k\ra\infty} \max_{\ell,m \in I_{\e,n_k}} \frac{|E_\w[T_m] - E_\w[T_\ell] - \frac{m-\ell}{\vp}|}{n_k^{1/4+\e/2+\e'}} 
 = 0, \quad P\text{-a.s.}
%, \quad \text{for } \e'>\frac{1/2-\e}{\k}. 
\end{equation}
Next, since $\gamma(1/2-\e)<1$ it follows that $n_{k+1}\vp - n_{k+1}^{1/2+\e} < n_k \vp + n_k^{1/2+\e}$ for $k$ large, so that $I_{\e,n_k} \cap I_{\e,n_{k+1}} \neq \emptyset$ for all $k$ large.
If $n_k \leq n < n_{k+1}$ and $I_{\e,n_k} \cap I_{\e,n_{k+1}}$ then it follows that 
\[
 \max_{\ell,m \in I_{\e,n}} \frac{|M_m-M_\ell|}{n^{1/4+\e/2+\e'}}
\leq  \max_{\ell,m \in I_{\e,n_k}\cup I_{\e,n_{k+1}}} \frac{|E_\w[T_m] - E_\w[T_\ell] - \frac{m-\ell}{\vp}|}{n_k^{1/4+\e/2+\e'}},
\]
and using \eqref{Mnfmdnk} and the fact that $n_{k+1}/n_k \to 1$ as $k\to \infty$ the right hand side vanishes almost surely as $k\to\infty$. Thus, we have shown that 
\[
 \lim_{n\ra\infty} \max_{k,\ell \in I_{\e,n}} \frac{ \left| E_\w[T_k]  - E_\w[T_\ell]-\frac{k-\ell}{v} \right| }{n^{1/4+\e/2+\e'}} = 0, \quad P\text{-a.s.,}\quad \text{for any } \e' > \frac{1/2-\e}{\k}. 
\]
Note that by taking $\e'$ arbitrarily close to $\frac{1/2-\e}{\k}$ this is equivalent to the statement \eqref{fmdas} we are trying to prove. 
\end{proof}

\begin{proof}[Proof of Proposition \ref{prop:qVmd1}]
Consider the martingale $\{L_n\}_{n\geq 0}$ defined by $L_0 = 0$ and 
\[
 L_n 
%= \Var_\w(T_n) - \sum_{k=0}^{n-1} E\left[ \Var_{\th^k \w}(\tau_1) \, | \, \mathcal{F}_{k-1} \right] 
%= \sum_{k=0}^{n-1} \left\{  \Var_{\th^k \w}(\tau_1) - E\left[ \Var_{\th^k \w}(\tau_1) \, | \, \mathcal{F}_{k-1} \right] \right\}.
= \Var_\w(T_n) - \sum_{k=0}^{n-1} E\left[ V_k \, | \, \mathcal{F}_{k-1} \right] 
= \sum_{k=0}^{n-1} \left(  V_k - E\left[ V_k \, | \, \mathcal{F}_{k-1} \right] \right).
\]
It follows from Lemma \ref{lem:taumtail} that $E[|V_0|^p] \leq E[|E_\w[\tau_1^2]|^p] < \infty$ for any $p < \k/2$, and thus Lemma \ref{lem:Mnmd} implies that 
\begin{equation}\label{Lnmd}
 \lim_{n\ra\infty} \frac{L_n}{n^{\frac{2}{4\wedge\k}+\e}} = 0, \quad P\text{-a.s.}, \quad \text{for any } \e>0. 
\end{equation}
To compare $L_n$ to $\Var_\w(T_n) - \s^2 n$ we need to give a different representation of $L_n$. 
To this end, it follows from the recursive formula for the quenched variance in \eqref{Vkrec} that 
\[
 E[V_k \, | \, \mathcal{F}_{k-1} ] = 
E\left[ (\rho_k+\rho_k^2)\left( 1 + \mu_{k-1} \right)^2 + \rho_{k} V_{k-1} \, \Bigl| \, \mathcal{F}_{k-1} \right]
= (r_1+r_2)(1+\mu_{k-1})^2 + r_1 V_{k-1},  
\]
and thus
\begin{align}
L_n  
&= \Var_\w(T_n) - \sum_{k=0}^{n-1} \left\{ (r_1+r_2)\left( 1+\mu_{k-1} \right)^2 + r_1 V_{k-1} \right\} \nonumber \\
&= \Var_\w(T_n) - (r_1+r_2)\left\{ n + 2 \sum_{k=-1}^{n-2} \mu_k + \sum_{k=-1}^{n-2} \mu_k^2 \right\} - r_1 \sum_{k=-1}^{n-2} V_k \nonumber \\
&= (1-r_1) \Var_\w(T_n) - (r_1+r_2)\left\{ n + 2 E_\w[T_n] + \sum_{k=0}^{n-1} \mu_k^2 \right\} \label{Vndec1} \\
&\qquad - (r_1+r_2)( 2\mu_{n-1}+\mu_{n-1}^2-2\mu_{-1}-\mu_{-1}^2 ) + r_1( V_{n-1}-V_{-1} ) \nonumber 
\end{align}
To further simplify this, note that it follows from \eqref{mumoments} and \eqref{Vmean} that 
\begin{align*}
&(1-r_1) \s^2 - (r_1+r_2)\left( 1 + 2E[\mu_0] + E[\mu_0^2] \right) \\
&=  \frac{ 4(r_1+r_2)(1 + r_1)}{(1-r_1)(1-r_2)} - (r_1+r_2)\left( 1 + \frac{2(1+r_1)}{1-r_1} + \frac{1+3r_1+3r_2+r_1r_2}{(1-r_1)(1-r_2)} \right) \\
%&= \frac{ 4(r_1+r_2)(1 + r_1)}{(1-r_1)(1-r_2)} - (r_1+r_2)\left( \frac{3+r_1}{1-r_1} + \frac{1+3r_1+3r_2+r_1r_2}{(1-r_1)(1-r_2)} \right) \\
%&= \frac{ 4(r_1+r_2)(1 + r_1)}{(1-r_1)(1-r_2)} - (r_1+r_2)\left( \frac{(3+r_1)(1-r_2)}{(1-r_1)(1-r_2)} + \frac{1+3r_1+3r_2+r_1r_2}{(1-r_1)(1-r_2)} \right) \\
%&= \frac{ 4(r_1+r_2)(1 + r_1)}{(1-r_1)(1-r_2)} - (r_1+r_2)\left( \frac{ 3-3r_2+r_1-r_1r_2 + 1+3r_1+3r_2+r_1r_2}{(1-r_1)(1-r_2)} \right) \\
%&= \frac{ 4(r_1+r_2)(1 + r_1)}{(1-r_1)(1-r_2)} - (r_1+r_2)\left( \frac{ 4 + 4r_1}{(1-r_1)(1-r_2)} \right) \\
&= \frac{ 4(r_1+r_2)(1 + r_1)}{(1-r_1)(1-r_2)} - \frac{ 4(r_1+r_2)(1 + r_1)}{(1-r_1)(1-r_2)} = 0. 
\end{align*}
Therefore, we have that 
\begin{align*}
 L_n
&= (1-r_1)\left( \Var_\w(T_n) -\s^2 \right) - (r_1+r_2)\left\{ 2\left( E_\w[T_n]-\frac{n}{\vp} \right) + \sum_{k=0}^{n-1} \left(\mu_k^2 - E[\mu_0^2] \right) \right\} \\
&\qquad - (r_1+r_2)( 2\mu_{n-1}+\mu_{n-1}^2-2\mu_{-1}-\mu_{-1}^2 ) + r_1( V_{n-1}-V_{-1} ) \nonumber 
\end{align*}
From this representation of $L_n$, by \eqref{Lnmd} and Lemma \ref{lem:qETmd} we see that to finish the proof of Proposition \ref{prop:qVmd1} it is enough to show that 
\begin{equation}\label{muk2md}
 \lim_{n\ra\infty} \frac{ \sum_{k=0}^{n-1} \left(\mu_k^2 - E[\mu_0^2] \right) }{ n^{\frac{2}{4\wedge \k}+\e}} = 0, \quad P\text{-a.s.}
\end{equation}
and 
\begin{equation}\label{taun2}
  \lim_{n\ra\infty} \frac{ \mu_{n-1}^2 + V_{n-1} }{ n^{\frac{2}{4\wedge \k}+\e}} =  \lim_{n\ra\infty} \frac{ E_{\th^{n-1}\w}[\tau_1^2] }{ n^{\frac{2}{4\wedge \k}+\e}} = 0, \quad P\text{-a.s.}
\end{equation}
To prove \eqref{taun2}, note that it follows from Lemma \ref{lem:taumtail} that 
\[
 P\left( E_{\th^{n-1}\w}[\tau_1^2] \geq \d \, n^{\frac{2}{4\wedge\k}+\e} \right) = \bigo( n^{-\frac{\k}{4\wedge\k} - \frac{\e\k}{4} } ) = \bigo( n^{-1-\frac{\e\k}{4}}),
\]
and then \eqref{taun2} follows from the Borel-Cantelli Lemma. 

It remains only to prove \eqref{muk2md}, and to do this we will use another martingale. Define $H_n=0$ and 
\[
 H_n = \sum_{k=0}^{n-1} \left\{ \mu_k^2 - E[\mu_k^2 \, | \, \mathcal{F}_{k-1} ] \right\}, \quad n\geq 1. 
\]
Note that Lemmas \ref{lem:taumtail} and \ref{lem:Mnmd} imply that for any $\e>0$, 
\begin{equation}\label{Hnmd}
 \lim_{n\ra\infty} \frac{H_n}{n^{\frac{2}{4\wedge \k}+ \e}} = 0, \quad P\text{-a.s.}
\end{equation}
To use this to prove \eqref{muk2md} we need to give a different representation of $H_n$. 
%Since {\color{red}it can be shown that $\mu_k = 1+\rho_k + \rho_k \mu_{k-1}$,} it follows that 
Using the recursive formula for $\mu_k$ in \eqref{mukrec} it follows that
\begin{align*}
 E[\mu_k^2 \, | \, \mathcal{F}_{k-1} ] 
&= E\left[ (1+\rho_k)^2 + 2\rho_k(1+\rho_k) \mu_{k-1} + \rho_k^2 \mu_{k-1}^2 \, | \, \mathcal{F}_{k-1} \right] \\
&= 1+2r_1+r_2 + 2(r_1+r_2) \mu_{k-1} + r_2 \mu_{k-1}^2. 
\end{align*}
Consequently, the martingale $H_n$ can be re-written as 
\begin{align*}
 H_n 
%&= \sum_{k=0}^{n-1} \left\{ \mu_k^2 - \left( 1+2r_1+r_2 + 2(r_1+r_2) \mu_{k-1} + r_2 \mu_{k-1}^2 \right) \right\} \\
&= \sum_{k=0}^{n-1} \mu_k^2 - (1+2r_1+r_2)n - 2(r_1+r_2)\sum_{k=-1}^{n-2} \mu_k - r_2 \sum_{k=-1}^{n-2} \mu_k^2 \\
&= (1-r_2) \sum_{k=0}^{n-1} \mu_k^2 - (1+2r_1+r_2)n - 2(r_1+r_2) E_\w[T_n] \\
&\qquad + 2(r_1+r_2)( \mu_{n-1} - \mu_{-1} ) + r_2( \mu_{n-1}^2 - \mu_{-1}^2 )
\end{align*}
Since the explicit formulas for $E[\mu_0]$ and $E[\mu_0^2]$ in \eqref{mumoments} imply that 
\begin{align*}
(1-r_2) E[\mu_0^2] - 2(r_1+r_2) E[\mu_0] 
%&=(1-r_2) \frac{1+3r_1+3r_2+r_1r_2}{(1-r_1)(1-r_2)} - 2(r_1+r_2) \frac{1+r_1}{1-r_1} \\
&= \frac{1+3r_1+3r_2+r_1r_2}{1-r_1} - 2(r_1+r_2) \frac{1+r_1}{1-r_1} \\
%&= \frac{1+3r_1+3r_2+r_1r_2}{1-r_1} - \frac{2r_1+2r_1^2+2r_2+2r_1r_2}{1-r_1} \\
&= \frac{1+r_1+r_2-2r_1^2 - r_1r_2}{1-r_1} \\
%&= \frac{(1-r_1)(1+2r_1+r_2)}{1-r_1} \\
&= 1+2r_1+r_2,
\end{align*}
we can further simplify the expression for $H_n$ as 
\begin{align}
 H_n 
&= (1-r_2) \sum_{k=0}^{n-1} ( \mu_k^2 - E[\mu_0^2] ) - 2(r_1+r_2) \left( E_\w[T_n] - \frac{n}{\vp} \right) \label{Hndec} \\
&\qquad + 2(r_1+r_2)( \mu_{n-1} - \mu_{-1} ) + r_2( \mu_{n-1}^2 - \mu_{-1}^2 ) \nonumber 
\end{align}
An argument similar to the proof of \eqref{taun2} shows that $\lim_{n\ra\infty} \frac{\mu_{n-1}^2}{n^{\frac{2}{4\wedge\k}+\e}} = 0$, $P$-a.s, and thus the proof of \eqref{muk2md} follows from applying  \eqref{Hnmd} and Lemma \ref{lem:qETmd} to \eqref{Hndec}.
\end{proof}

\section{Quenched CLT rates of convergence for hitting times}\label{sec:Trates}

Since the hitting times $T_n= \sum_{k=1}^n \tau_k$ are the sum of random variables that are independent under the quenched measure, a key element in our proof of Theorems \ref{th:BETn} and \ref{th:BETnds} will be the following generalization of the Berry-Esseen estimates. 

\begin{thm}[Theorem V.3.6 in \cite{petrovRW}]\label{th:petrov}
Let $S_n = \sum_{k=1}^n \xi_i$ be the sum of independent zero mean random variable with finite variance. 
For any $\d \in (0,1]$ there exists a universal constant $A_\delta>0$ such that 
%if $E[|\xi_1|^{2+\d}] < \infty$ for each $1\leq i \leq n$, then 
\[
 \sup_x \left| P\left( \frac{S_n}{\sqrt{\Var(S_n)}} \leq x \right) - \Phi(x) \right| \leq \frac{A_\delta}{\Var(S_n)^{1+\frac{\d}{2}}} \sum_{k=1}^n E\left[|\xi_i|^{2+\d}\right]. 
\]
% If $S_n = \sum_{k=1}^n \xi_i$ is the sum of independent zero mean random variable with $E[|\xi_i|^3]<\infty$ for each $i\geq 1$, then there exists a universal constant $A>0$ such that 
%\[
% \sup_x \left| P\left( \frac{S_n}{\sqrt{\Var(S_n)}} \leq x \right) - \Phi(x) \right| \leq \frac{A}{\Var(S_n)^{3/2}} \sum_{k=1}^n E[|\xi_i|^3]. 
%\]
\end{thm}

\begin{proof}[Proof of Theorem \ref{th:BETn}]
Since under the quenched measure $T_n - E_\w[T_n] = \sum_{k=1}^n (\tau_k - E_\w[\tau_k])$ is the sum of independent zero mean random variables, 
it follows immediately from Theorem \ref{th:petrov} (with $\d=1$) that 
\begin{equation}\label{bound1}
 \sup_x \left| \overline{F}_{n,\w}(x) - \Phi(x) \right| \leq \frac{A_1}{\Var_\w(T_n)^{3/2}} \sum_{k=1}^n E_\w\left[ \left| \tau_k - E_\w[\tau_k] \right|^3 \right].
\end{equation}
Since $\Var_\w(T_n)/n \to \s^2$ almost surely as $n\to\infty$ we need only to consider the asymptotics of the last sum on the right. The analysis is different in the cases $\k>3$ and $\k \in (2,3]$. 

\textbf{Case I: $\k>3$.}
In this case it follows from Lemma \ref{lem:taumtail} that 
%$E\left[E_\w[|\tau_1 - E_\w[\tau_1]|^3] \right] = \E[|\tau_1 - E_\w[\tau_1]|^3] < \infty$. 
$\E[|\tau_1 - E_\w[\tau_1]|^3] < \infty$. 
Therefore, Birkhoff's Ergodic Theorem implies that 
\[
 \lim_{n\to\infty} \frac{1}{n} \sum_{k=1}^n E_\w\left[ \left| \tau_k - E_\w[\tau_k] \right|^3 \right]  = \E[|\tau_1 - E_\w[\tau_1]|^3].
\]
Applying this to \eqref{bound1} we obtain that 
\begin{align*}
 \limsup_{n\to\infty} \sqrt{n} \sup_x \left| \overline{F}_{n,\w}(x) - \Phi(x) \right| 
&\leq 
\lim_{n\to\infty} \frac{A_1 \sqrt{n}}{\Var_\w(T_n)^{3/2}} \sum_{k=1}^n E_\w\left[ \left| \tau_k - E_\w[\tau_k] \right|^3 \right] \\
%&= A_1 \left( \frac{n}{\Var_\w(T_n)} \right)^{3/2} \frac{1}{n} \sum_{k=1}^n E_\w\left[ \left| \tau_k - E_\w[\tau_k] \right|^3 \right] \\
&= \frac{A_1 \E[|\tau_1 - E_\w[\tau_1]|^3]}{\s^3}. 
\end{align*}

\textbf{Case II: $\k \in (2,3]$.} It follows from Lemma \ref{lem:taumtail} that for any $p < \k/3$, 
\begin{align*}
 E\left[ \left( E_\w[ | \tau_1 - E_\w[\tau_1] |^3 ] \right)^p \right] 
&\leq
% E\left[ \left( 4 E_\w[\tau_1^3] + 4 (E_\w[\tau_1])^3 \right)^p \right]= 
4^p E\left[ \left( E_\w[\tau_1^3] + (E_\w[\tau_1])^3 \right)^p \right] \\
&\leq 4^p 2^{p-1} E\left[ E_\w[\tau_1^3]^p + E_\w[\tau_1]^{3p} \right] < \infty. 
\end{align*}
Since the quenched expectations $E_\w\left[ \left| \tau_k - E_\w[\tau_k] \right|^3 \right]$ are an ergodic sequence in $k$, it follows that if $p < \frac{\k}{3} \leq 1$ then 
\begin{align*}
\limsup_{n\to\infty} \frac{1}{n^{1/p}} \sum_{k=1}^n E_\w\left[ \left| \tau_k - E_\w[\tau_k] \right|^3 \right] 
&= \limsup_{n\to\infty} \left\{ \frac{1}{n} \left( \sum_{k=1}^n E_\w\left[ \left| \tau_k - E_\w[\tau_k] \right|^3 \right] \right)^p  \right\}^{1/p} \\
&\leq \lim_{n\to\infty} \left\{ \frac{1}{n}  \sum_{k=1}^n \left( E_\w\left[ \left| \tau_k - E_\w[\tau_k] \right|^3 \right] \right)^p  \right\}^{1/p} \\
&= \left\{ E\left[ \left( E_\w\left[ \left| \tau_1 - E_\w[\tau_1] \right|^3 \right] \right)^p \right] \right\}^{1/p}< \infty, \quad P-a.s.
\end{align*}
By taking $p$ arbitrarily close to $\k/3$ we can therefore conclude that
\[
 \lim_{n\to\infty} \frac{1}{n^{\frac{3}{\k}+\e}} \sum_{k=1}^{n} E_\w\left[ \left| \tau_k - E_\w[\tau_k] \right|^3 \right] = 0, \quad P-a.s.
\]
Applying this to \eqref{bound1} we obtain that for any $\e>0$, 
\begin{align*}
 \limsup_{n\to\infty} n^{\frac{3}{2}-\frac{3}{\k}-\e} \sup_x \left| \overline{F}_{n,\w}(x) - \Phi(x) \right| 
&\leq A_1 \left( \frac{n}{\Var_\w(T_n)} \right)^{3/2} \frac{1}{n^{\frac{3}{\k}+\e}} \sum_{k=1}^n E_\w\left[ \left| \tau_k - E_\w[\tau_k] \right|^3 \right] = 0, \quad P\text{-a.s.}
\end{align*}
\end{proof}

\begin{rem}
% Since $E[E_\w[\tau_1^3]] = \infty$ when $\k < 3$ 
In the case of $\k \in (2,3]$
one might wonder if better rates of convergence could be obtained by applying Theorem \ref{th:petrov} with $2+\d < \k$. However, it's easy to see that this only gives $n^{\frac{\k}{2}-1-\e} \|\overline{F}_{n,\w} - \Phi \|_\infty \to 0$ for any $\e>0$, and since $\frac{\k}{2}-1 < \frac{3}{2}-\frac{3}{\k}$ when $\k \in (2,3)$ the bounds in the statement of Theorem \ref{th:BETn} are better.
\end{rem}

\begin{proof}[Proof of Theorem \ref{th:BETnds}]
Since 
\[
 F_{n,\w}(x) 
%=  P_\w\left(  \frac{T_n - E_\w[T_n]}{\s\sqrt{n}} \leq x \right)
= P_\w\left(  \frac{T_n - E_\w[T_n]}{\sqrt{\Var_\w(T_n)}} \leq x \sqrt{ \frac{\s^2 n }{\Var_\w(T_n)}} \right) 
= \overline{F}_{n,\w}\left( x \sqrt{ \frac{\s^2 n }{\Var_\w(T_n)}} \right), 
\]
we note that 
\begin{equation}
 \sup_x \left| F_{n,\w}(x) - \Phi(x) \right| 
\leq \sup_x | \overline{F}_{n,\w}(x) - \Phi(x) | + \sup_x \left| \Phi\left( x \sqrt{ \frac{\s^2 n }{\Var_\w(T_n)}} \right) - \Phi(x) \right|
\end{equation}
The first term on the right can be controlled by Theorem \ref{th:BETn},
while for the second term on the right we note (see for instance \cite[Section V.3, equation (3.3)]{petrovRW}) that 
\[
 \sup_x \left| \Phi(x) - \Phi(a x) \right| \leq \begin{cases} \frac{1}{\sqrt{2\pi e}} \frac{1-a}{a} & \text{if } a\in (0,1) \\ \frac{1}{\sqrt{2\pi e}} (a-1) & \text{if } a \geq 1. \end{cases}
\]
%This bound can be improved, however, by noting that the uniform upper bound $\left| \Phi(x) - \Phi(a x) \right| \leq \frac{1}{2}$ holds for all $x$ and $a>0$. 
%%%%Since $\phi(1) \frac{1-a}{a} > 1/2$ when $a < \frac{2\phi(1)}{2\phi(1)+1}$, it follows that we can get the statement of the lemma with $C = \phi(1)+\frac{1}{2} \approx 0.741971...$. 
%Since $\phi(1) (a-1) > 1/2$ when $a>1+\frac{1}{2\phi(1)}$, it follows that we can get the statement of the lemma with $C = \phi(1)+\frac{1}{2} \approx 0.741971...$.
It follows from Proposition \ref{prop:qVmd1} that for any $\e>0$,  $P$-a.e.\ environment $\w$,
\[
 \sqrt{ \frac{\s^2 n }{\Var_\w(T_n)}} = 1 + o\left( n^{\frac{2}{4\wedge\k}+\e -1} \right), \quad \text{for $P$-a.e.\ environment $\w$,}
\]
and therefore
\[
 \lim_{n\to\infty} n^{1-\frac{2}{4\wedge\k} - \e} \sup_x \left| \Phi\left( x \sqrt{ \frac{\s^2 n }{\Var_\w(T_n)}} \right) - \Phi(x) \right| = 0, \quad P\text{-a.s.}
\]
Since in all cases the rate of decay of the first term on the right in \eqref{bound1} given by Theorem \ref{th:BETn} decays faster than $n^{-1+\frac{2}{4\wedge\k}+\e}$ this completes the proof of Theorem \ref{th:BETnds}. 
\end{proof}

\section{Quenched CLT rates of convergence for the walk}\label{sec:Xrates}

As noted in the introduction, we will obtain rates of convergence for the quenched CLT for $X_n$ from the rates of the quenched CLT for $T_n$ in Theorem \ref{th:BETn}. 
The transfer of limiting distributions from hitting times to the position of the walk hinges on the fact that $P_\w(T_k > n) = P_\w(X_n^* < k)$ where $X_n^* = \max_{k\leq n} X_k$ is the running maximum of the walk up to time $n$. 
In preparation for the proof of Theorem \ref{th:BEXn} we will first prove the following Lemma which will allow us to compare the distribution of $X_n^*$ and $X_n$. 

\begin{lem}\label{XnXns}
 If $\k>0$ then there exists a constant $B>0$ such that $P_\w( X_n^* - X_n \geq B \log n) \leq \frac{1}{\sqrt{n}}$ for $P$-a.e.\ environment $\w$ and for all $n$ sufficiently large. 
\end{lem}
\begin{proof}
It was shown in \cite{gsMVSS} that if $\k>0$ then $\P(T_{-m} < \infty) \leq C_1 e^{-C_2 m}$ for some constants $C_1,C_2 > 0$ and all $m\geq 1$. 
It follows from this that 
\begin{align*}
 \P(X_n^* - X_n \geq m) 
&\leq \sum_{k=0}^{n-1} \P\left( \inf_{i > T_k} X_i \leq k-m \right) \\
&= \sum_{k=0}^{n-1} \P^k(T_{k-m} < \infty)
= n \P(T_{-m}<\infty) \leq C_1 n e^{-C_2 m}. 
\end{align*}
Therefore, by Chebychev's inequality we have 
\[
 P\left( P_\w( X_n^* - X_n \geq B \log n) > \frac{1}{\sqrt{n}} \right)
\leq \sqrt{n} \P( X_n^* - X_n \geq B \log n ) \leq C_1 n^{3/2} e^{-C_2 B \log n}.
\]
%\begin{align*}
% P\left( P_\w( X_n^* - X_n > B \log n) > \frac{1}{\sqrt{n}} \right)
%&\leq \sqrt{n} \P( X_n^* - X_n > B \log n ) \\
%&\leq n^{3/2} \P( T_{-B \log n} < \infty ) \\
%&\leq n^{3/2} C_1 e^{-C_2 B \log n}. 
%\end{align*}
If $B > \frac{5}{2 C_2}$, then this bound is summable and the conclusion of the lemma follows from the Borel-Cantelli Lemma. 
\end{proof}

\begin{proof}[Proof of Theorem \ref{th:BEXn}]
Since the distribution function $\Phi(x)$ is continuous, rates of convergence for $G_{n,\w}$ are equivalent to rates of convergence for 
\[
 G_{n,\w}^\circ(x) = \lim_{\e\ra 0^+} G_{n,\w}(x+\e) = P_\w\left( \frac{X_n - n\vp + Z_n(\w)}{\s \vp^{3/2}\sqrt{n}} < x \right).
\]
Since it is more convenient for the proof, we will prove rates of convergence for $G_{n,\w}^\circ$. 
In fact, the strategy of the proof will be to first prove rates of convergence for 
\[
G_{n,\w}^*(x) = P_\w\left( \frac{X_n^* - n\vp + Z_n(\w)}{\s \vp^{3/2}\sqrt{n}} < x \right)
\]
and then use Lemma \ref{XnXns} to obtain corresponding rates of convergence for $G_{n,\w}^\circ$. 
Indeed, since 
\begin{align*}
& \left| G_{n,\w}^\circ(x) - \Phi(x)\right| \\
%&\quad \leq \left| G_{n,\w}^\circ(x) - G_{n,\w}^*\left(x + \frac{B\log n}{\sqrt{n}} \right) \right| + \left| G_{n,\w}^*\left(x + \frac{B\log n}{\sqrt{n}} \right) - \Phi\left( x + \frac{B\log n}{\sqrt{n}} \right) \right|\\
%&\quad \qquad  +  \left| \Phi\left( x + \frac{B\log n}{\sqrt{n}} \right) - \Phi(x) \right|  \\
&\quad \leq \left| G_{n,\w}^\circ(x) - G_{n,\w}^*\left(x + \frac{B\log n}{\s \vp^{3/2} \sqrt{n}} \right) \right| + \sup_{y\in\R} \left| G_{n,\w}^*\left(y\right) - \Phi\left( y \right) \right| +  \left| \Phi\left( x + \frac{B\log n}{\s \vp^{3/2} \sqrt{n}} \right) - \Phi(x) \right| \\
&\quad \leq P_\w (X_n^* - X_n \geq B \log n) + \sup_{y\in\R} \left| G_{n,\w}^*(y) - \Phi(y) \right| +  \frac{B\log n}{\s \vp^{3/2} \sqrt{2 \pi n}} ,
\end{align*}
it follows from Lemma \ref{XnXns} that to prove the almost sure convergence rate of convergence in Theorem \ref{th:BEXn} we need only to show 
\begin{equation}\label{BETXns}
 \lim_{n\to\infty} n^{\frac{1}{4} - \frac{1}{2\k} - \e} \sup_{x\in\R} \left| G_{n,\w}^*(x) - \Phi(x) \right| = 0, \quad P\text{-a.s.}, \quad \text{for  any } \e>0.
\end{equation}

For the proof of \eqref{BETXns} we begin by noting that since $P_\w(X_n^* < k ) = P_\w(T_k > n)$ for any $n,k\geq 1$ that
\begin{equation}\label{XnsTn}
 G_{n,\w}^*(x) = P_\w\left( X_n^* < n\vp - Z_n(\w) + x \s \vp^{3/2} \sqrt{n} \right) = P_\w\left( T_{k(n,\w,x)} > n \right) 
\end{equation}
whenever 
\[
 k(n,\w,x) := \left\lceil n\vp - Z_n(\w) + x\s \vp^{3/2} \sqrt{n} \right\rceil \geq 1. 
\]
%If $k(n,\w,x) \leq \frac{n\vp}{2}$ then it follows from \cite{gzQSETE} that 
%\[
% P_\w\left( T_{k(n,\w,x)} > n \right) \leq P_\w\left( T_{\fl{n\vp/2}} > n \right) \leq e^{-\sqrt{n}}, 
%\]
Throughout the remainder of our proof, we will fix an arbitrary $\e \in (0,1/2)$. 
Let $x_{n,\e}^-=x_{n,\e}^-(\w)$ and $x_{n,\e}^+=x_{n,\e}^+(\w)$ be such that $k(n,\w,x_{n,\e}^-) = \ceil{n\vp - n^{1/2+\e}}$ and $k(n,\w,x_{n,\e}^+) = \fl{n\vp + n^{1/2+\e}}$. 
We will use \eqref{XnsTn} and Theorem \ref{th:BETn} to control $|G_{n,\w}^*(x)-\Phi(x)|$ but our analysis will be different depending on whether or not $x \in [x_{n,\e}^-,x_{n,\e}^+]$.

\textbf{Case I: $x\in [x_{n,\e}^-,x_{n,\e}^+]$.} In this case, it follows from \eqref{XnsTn} that 
\begin{align}
\left| G_{n,\w}^*(x) - \Phi(x) \right| 
&= \left| P_\w\left( \frac{ T_{k(n,\w,x)} - E_\w[T_{k(n,\w,x)}] }{ \sqrt{\Var_\w(T_{k(n,\w,x)})} } > \frac{ n - E_\w[T_{k(n,\w,x)}] }{ \sqrt{\Var_\w(T_{k(n,\w,x)})} } \right) - \Phi(x) \right| \nonumber  \\
&= \left|\overline{F}_{k(n,\w,x),\w}\left( \frac{ n - E_\w[T_{k(n,\w,x)}] }{ \sqrt{\Var_\w(T_{k(n,\w,x)})} } \right) - \Phi(-x) \right| \nonumber \\
&\leq \sup_{t\in\R} \left| \overline{F}_{k(n,\w,x),\w}(t) - \Phi(t) \right| + \left| \Phi\left( \frac{ n - E_\w[T_{k(n,\w,x)}] }{ \sqrt{\Var_\w(T_{k(n,\w,x)})} } \right)  - \Phi(-x) \right| \nonumber \\ 
%&\leq \sup_{|m -n\vp|\leq n^{1/2+\e}} \sup_{t\in\R} \left| \overline{F}_{m,\w}(t) - \Phi(t) \right| + \frac{1}{\sqrt{2\pi}} \left| \frac{E_\w[T_{k(n,\w,x)}] - n }{ \sqrt{\Var_\w(T_{k(n,\w,x)})} } - x \right|. \label{GnPhi1}
&\leq \sup_{|m -n\vp|\leq n^{1/2+\e}} \left\| \overline{F}_{m,\w} - \Phi \right\|_\infty + \frac{1}{\sqrt{2\pi}} \left| \frac{E_\w[T_{k(n,\w,x)}] - n }{ \sqrt{\Var_\w(T_{k(n,\w,x)})} } - x \right|. \label{GnPhi1}
%&\leq o\left( n^{-1 + \frac{2}{4\wedge\k}+\e' } \right) + \frac{1}{\sqrt{2\pi}} \left| \frac{E_\w[T_{k(n,\w,x)}] - n }{ \sqrt{\Var_\w(T_{k(n,\w,x)})} } - x \right|, \label{GnPhi1}
\end{align}
The first term in \eqref{GnPhi1} can be controlled by Theorem \ref{th:BETn}.
For the second term in \eqref{GnPhi1}, note first of all that (recalling the definition of $Z_n(\w)$ from the statement of Theorem \ref{th:qclt})
\begin{align*}
 n &= E_\w[T_{\fl{n\vp}}] - \left( E_\w[T_{\fl{n\vp}}] - \frac{\fl{n\vp}}{\vp} \right) + \frac{n\vp-\fl{n\vp}}{\vp} \\
&= E_\w[T_{\fl{n\vp}}] - \frac{Z_n(\w)}{\vp} + \bigo(1),
\end{align*}
where here (and below) we will use $\bigo(1)$ to denote uniformly bounded error terms coming from integer rounding. 
Therefore, 
\begin{align*}
 E_\w[T_{k(n,\w,x)}] - n 
&= E_\w[T_{k(n,\w,x)}] - E_\w[T_{\fl{n\vp}}] + \frac{Z_n(\w)}{\vp} + \bigo(1) \\
&= \left( E_\w[T_{k(n,\w,x)}] - E_\w[T_{\fl{n\vp}}] - \frac{k(n,\w,x)-n\vp }{\vp}\right) + \frac{k(n,\w,x)-n\vp + Z_n(\w)}{\vp} + \bigo(1) \\
&= \left( E_\w[T_{k(n,\w,x)}] - E_\w[T_{\fl{n\vp}}] - \frac{k(n,\w,x)-n\vp }{\vp} \right) + x\s \sqrt{n\vp} + \bigo(1),
\end{align*}
where the last equality follows from the definition of $k(n,\w,x)$. 
Since $x \in [x_{n,\e}^-,x_{n,\e}^+]$ implies that $k(n,\w,x) \in I_{\e,n} = [n\vp-n^{1/2+\e},n\vp+n^{1/2+\e}]$ it follows from Proposition \ref{prop:qETfmd} that the first term in the last line is bounded (uniformly over $x\in[x_{n,\e}^-,x_{n,\e}^+]$) by something that is $o\left( n^{\frac{1}{4} + \frac{1}{2\k} + \e(\frac{1}{2}-\frac{1}{\k})+\e'} \right)$ for any $\e'>0$. 
Finally, we claim that $\Var_\w(T_{k(n,\w,x)})$ is asymptotically close to $\s^2 \vp n$ uniformly over $x \in [x_{n,\e}^-,x_{n,\e}^+]$. 
Indeed, since
%Finally, since for any $x \in [x_{n,\e}^-,x_{n,\e}^+]$ we have  
$\Var_\w(T_{n\vp - n^{1/2+\e}}) \leq \Var_\w(T_{k(n,\w,x)}) \leq \Var_\w(T_{n\vp + n^{1/2+\e}})$ it follows from the fact that $\Var_\w(T_m)\sim \s^2 m$ that 
\[
 \lim_{n\to\infty} \sup_{x\in [x_{n,\e}^-,x_{n,\e}^+] } \left| \frac{\Var_\w(T_{k(n,\w,x)})}{\s^2 n \vp} - 1 \right| = 0, \quad P\text{-a.s.} 
\]
We have therefore shown that for any $\e'>0$, 
\begin{equation}\label{ETkas}
 \lim_{n\to\infty} n^{\frac{1}{4} - \frac{1}{2\k} - \e(\frac{1}{2}-\frac{1}{\k})-\e'} \sup_{x \in [x_{n,\e}^-,x_{n,\e}^+]}  \left| \frac{E_\w[T_{k(n,\w,x)}] - n }{  \sqrt{\Var_\w(T_{k(n,\w,x)})} } - x \right| = 0, \quad P\text{-a.s.}
\end{equation}
Since Theorem \ref{th:BETn} implies that the first term in \eqref{GnPhi1} decays strictly faster than $n^{-\frac{1}{4}+\frac{1}{2\k}}$, we can conclude that 
%Combining this, together with \eqref{GnPhi1} and Theorem \ref{th:BETnds} and the fact that $\frac{1}{4} - \frac{1}{2\k} < 1- \frac{2}{4\wedge\k}$ if $\k>2$, we obtain that 
\begin{equation}\label{xsmall}
 \lim_{n\to\infty} n^{\frac{1}{4} - \frac{1}{2\k} - \e(\frac{1}{2}-\frac{1}{\k})-\e'} \sup_{x \in [x_{n,\e}^-,x_{n,\e}^+]} \left| G_{n,\w}^*(x) - \Phi(x) \right| = 0, \quad P\text{-a.s.} 
\end{equation}

\textbf{Case II: $x \notin [ x_{n,\e}^-, x_{n,\e}^+]$.} Since Lemma \ref{lem:qETmd} implies that $Z_n(\w)/n^{1/2+\e} \to 0$, it follows that for $n$ large enough (depending on $\w$) $x_{n,\e}^- < -n^{-\e/2}$ and $x_{n,\e}^+ > n^{\e/2}$. 
Therefore, by the monotonicity of the distribution functions we have 
\begin{align}
 \sup_{x < x_{n,\e}^-} \left| G_{n,\w}^*(x) - \Phi(x) \right| 
&\leq G_{n,\w}^*(x_{n,\e}^-) + \Phi(x_{n,\e}^-) \nonumber \\
&\leq \left| G_{n,\w}^*(x_{n,\e}^-) - \Phi(x_{n,\e}^-) \right| + 2 \Phi(x_{n,\e}^-) \nonumber \\
&\leq \left| G_{n,\w}^*(x_{n,\e}^-) - \Phi(x_{n,\e}^-) \right| + 2 \Phi(-n^{\e/2}), \label{negx}
\end{align}
and similarly
\begin{align}
 \sup_{x > x_{n,\e}^-} \left| G_{n,\w}^*(x) - \Phi(x) \right| 
&= \sup_{x > x_{n,\e}^-} \left|\left(1- G_{n,\w}^*(x)\right) - (1-\Phi(x)) \right| \nonumber \\
&\leq  1-G_{n,\w}^*(x_{n,\e}^+) + 1-\Phi(x_{n,\e}^+) \nonumber \\
%&\leq \left| ( 1-G_{n,\w}^*(x_{n,\e}^+) )  - ( 1-\Phi(x_{n,\e}^+) )  \right| + 2 ( 1-\Phi(x_{n,\e}^+) ) \nonumber \\
%&\leq \left| G_{n,\w}^*(x_{n,\e}^+) - \Phi(x_{n,\e}^+)   \right| + 2 ( 1-\Phi(x_{n,\e}^+) ) \nonumber \\
&\leq \left| G_{n,\w}^*(x_{n,\e}^+) - \Phi(x_{n,\e}^+)   \right| + 2 ( 1-\Phi(n^{\e/2}) ). \label{posx}
\end{align}
Since $\Phi(-n^{\e/2}) = 1-\Phi(n^{\e/2})$ decays faster than any polynomial in $n$, applying \eqref{xsmall} to \eqref{negx} and \eqref{posx} we obtain that 
\[
 \lim_{n\to\infty} n^{\frac{1}{4} - \frac{1}{2\k} - \e(\frac{1}{2}-\frac{1}{\k})-\e'} \sup_{x \in\R} \left| G_{n,\w}^*(x) - \Phi(x) \right| = 0, \quad P\text{-a.s.} 
\]
Finally, note that since $\e,\e'>0$ were arbitrary this completes the proof of the almost sure rate of convergence in Theorem \ref{th:BEXn}. 

The proof of the weaker in probability rates of convergence for $G_n$ in \eqref{BEXnip1} and \eqref{BEXnip2} are almost the same as the above proof of the almost sure convergence rates. The only difference is that instead of \eqref{ETkas}, the convergence in probability statement in Proposition \ref{prop:qETfmd} gives that for any $\e'>0$
\[
  \lim_{n\to\infty} n^{\frac{1}{4} - \frac{\e}{2} - \e'} \sup_{x \in [x_{n,\e}^-,x_{n,\e}^+]}  \left| \frac{E_\w[T_{k(n,\w,x)}] - n }{ \sqrt{\Var_\w(T_{k(n,\w,x)})} } - x \right| = 0, \quad \text{in $P$-probability.}
\]
The rest of the proof is essentially the same with the exception that when $\k \in (2,\frac{12}{5})$ and $\e>0$ is sufficiently small the dominant term in \eqref{GnPhi1} is the first term which by Theorem \ref{th:BETn} is $o(n^{-\frac{3}{2}+\frac{3}{\k} + \e''})$ for any $\e''>0$. 
%The rest of the proof is essentially the same with the exception that when $\k \in (2,\frac{8}{3})$ the dominant term in \eqref{GnPhi1} is the first term since $1-\frac{2}{\k} < \frac{1}{4}$. 
\end{proof}

%\bibliographystyle{alpha}
%\bibliography{RWREref}

\end{document}